\documentclass{amsart}
\usepackage{amsmath}
\usepackage{paralist}
\usepackage{graphics} 
\usepackage{epsfig} 
\usepackage{graphicx}  \usepackage{epstopdf}
\usepackage{amssymb}
\usepackage{amsthm}
\usepackage{epstopdf}
\usepackage{verbatim}
\usepackage{mathrsfs}
\usepackage{mathtools}
\usepackage{pstricks}
\usepackage{cleveref}
\usepackage{relsize}
\usepackage{tikz}
\usetikzlibrary{matrix}
\usepackage{subcaption}
\usepackage{pgfplots}
\usepackage{fixltx2e}
\usepackage{enumitem}
\usepackage{bm}
\usepackage{appendix}
\theoremstyle{plain}
\newtheorem{theorem}{Theorem}[section]
\newtheorem{lemma}[theorem]{Lemma}

\newtheorem{proposition}[theorem]{Proposition}

\theoremstyle{definition}
\newtheorem{definition}[theorem]{Definition}

\newtheorem{example}[theorem]{Example}

\newtheorem{remark}[theorem]{Remark}

\theoremstyle{plain}

\newcommand{\x}{\mathbf{x}}

\newcommand{\R}{\mathbb{R}}
\crefname{equation}{}{} 
\Crefname{equation}{}{} 
\crefname{theorem}{theorem}{theorems} 
\Crefname{theorem}{Theorem}{Theorems} 
\numberwithin{equation}{section}
\begin{document}
\title[Turnpike property and Hamilton-Jacobi-Bellman equation]{The turnpike property and the long-time behavior of the Hamilton-Jacobi-Bellman equation for finite-dimensional LQ control problems}
\thanks{This project has received funding from the European Research Council (ERC) under the European Union’s Horizon 2020 research and innovation programme (grant agreement NO. 694126-DyCon). 
The work of E.Z. is partially funded by the Alexander von Humboldt-Professorship program, the European Unions Horizon 2020 research and innovation programme under the Marie Sklodowska-Curie grant agreement No.765579-ConFlex, 
the Grant MTM2017-92996-C2-1-R COSNET of MINECO (Spain), the Air Force Office of Scientific Research (AFOSR) under Award NO. FA9550-18-1-0242.  and the Transregio 154 Project ‘‘Mathematical Modelling, Simulation and Optimization Using the Example of Gas Networks’’ of the German DFG}

\author{Carlos Esteve}
\address {Carlos Esteve, Dario Pighin \newline \indent
{Departamento de Matem\'aticas, \newline \indent
Universidad Aut\'onoma de Madrid},
\newline \indent
{28049 Madrid, Spain}
\newline \indent \hspace{2.5cm} \text{and} \newline \indent
{Chair of Computational Mathematics, Fundaci\'on Deusto}
\newline \indent
{Av. de las Universidades, 24}
\newline \indent
{48007 Bilbao, Basque Country, Spain}
}
\email{\texttt{carlos.esteve@uam.es}, \texttt{dario.pighin@uam.es}}

\author{Hicham Kouhkouh}
\address {Hicham Kouhkouh \newline \indent
{Dipartimento di Matematica, \newline \indent
Universit\`a di Padova},
\newline \indent
{via Trieste, 63; I-35121 Padova, Italy}
}
\email{\texttt{kouhkouh@math.unipd.it}}

\author{Dario Pighin}

\author{Enrique Zuazua}
\address{Enrique Zuazua \newline \indent
{Chair in Applied Analysis, Alexander von Humboldt-Professorship
\newline \indent
Department of Data Science, \newline \indent
Friedrich-Alexander-Universit\"at Erlangen-N\"urnberg}
\newline \indent
{91058 Erlangen, Germany}
\newline \indent \hspace{2.5cm} \text{and} \newline \indent
{Chair of Computational Mathematics, Fundación Deusto}
\newline \indent
{Av. de las Universidades, 24}
\newline \indent
{48007 Bilbao, Basque Country, Spain}
\newline \indent \hspace{2.5cm} \text{and} \newline \indent
{Departamento de Matemáticas, \newline \indent
Universidad Autónoma de Madrid,}
\newline \indent
{28049 Madrid, Spain}
}
\email{\texttt{enrique.zuazua@fau.de}}

\date{\today}
\begin{abstract}
	We analyze the consequences that the so-called turnpike property has on the long-time behavior of the value function corresponding to a finite-dimensional linear-quadratic optimal control problem with general terminal cost and constrained controls.
    We prove that, when the time horizon $T$ tends to infinity, the value function asymptotically behaves as $W(x) + c\, T + \lambda$, and we provide a control interpretation of each of these three terms, making clear the link with the turnpike property.
	As a by-product, we obtain the long-time behavior of the solution to the associated Hamilton-Jacobi-Bellman equation in a case where the Hamiltonian is not coercive in the momentum variable. As a result of independent interest, we showed that linear-quadratic optimal control problems with constrained control enjoy a turnpike property, also particularly when the steady optimum may saturate the control constraints. 
\end{abstract}

\maketitle

\section{Introduction}

\subsection{Motivation and setting}


We are interested in the asymptotic  behavior of the value function associated to an optimal control problem, when the time-horizon tends to infinity.
In particular, we want to deduce it as a consequence of a property that is satisfied by a large class of optimal control problems and arises when the time horizon is considered to be sufficiently large.
This is the so-called \emph{turnpike property}, which establishes that the optimal strategy in a controlled system during a sufficiently long time interval is to quickly stabilize from the initial state to the steady optimal one and not leave the latter until the time is close to the end. Morevoer, as a by-product of our study, we obtain the long-time behavior of the associated Hamilton-Jacobi-Bellman (HJB for short) equation in a case where the Hamiltonian is not strictly convex and not even coercive, a scenario much less considered in the literature.
This indicates that, in some cases, this kind of assumptions on the structure of the Hamiltonian can be merely relaxed to weaker assumptions concerning the controllability and observability of the optimal control problem.





Let us introduce the mathematical framework that we will use throughout the paper. We denote by $\mathcal{M}_{n,m}(\R)$ (resp. $\mathcal{M}_{n}(\R)$) the set of matrices over $\R$ with $n$ rows and $m$ columns (resp. $n$ rows and columns).
We consider the following optimal control problem in the finite-dimensional linear-quadratic setting, with constrained controls:
for a given time horizon $T>0$ and an initial state $x\in \R^n$,
we denote the trajectory of the system by $y (\cdot)$, which is determined by the solution to the following controlled linear ODE:
\begin{equation}\label{eq: linear ODE}
\begin{array}{ll}
\dot{y}(s) = A\, y(s) + B\, u(s), &  s\in (0,T) \\
y (0) = x,
\end{array}
\end{equation}
where $A\in \mathcal{M}_n  (\R)$, $B\in \mathcal{M}_{n,m} (\R)$, with $n,m\geq 1$, are two given matrices,
and $u$, that will be referred to as the control, can be any function in the set of admissible controls $\mathcal{U}_T \coloneqq  L^2(0,T; U)$, i.e. square-integrable functions $[0,T]\to U$, where $U\subseteq\R^m$ is a given nonempty closed and convex set, which can be either bounded or not.

The optimal control problem is to minimize, over the admissible controls $u\in \mathcal{U}_T$, the cost functional
\begin{equation}\label{eq: functional}
J_{T,x} (u) \coloneqq \dfrac{1}{2} \int_0^T \left[\| u(s)\|^2 + \| C \, y (s)-z\|^2\right] ds  + g(y(T)), 
\end{equation}
where $C\in \mathcal{M}_n(\R)$ is a given matrix, $z\in \R^n$ is the prescribed \emph{target} and $g:\R^n\to \R$ is a given locally Lipschitz function bounded from below, known as the \emph{final cost}.  The value function associated to the optimal control problem \eqref{eq: linear ODE}-\eqref{eq: functional} is defined as
\begin{equation}\label{eq: value function}
V (x,T) \coloneqq \inf_{u\in \mathcal{U}_T}  J_{T,x}(u),\quad \text{s.t.}\; \eqref{eq: linear ODE}.
\end{equation}
We also consider the associated stationary problem, consisting in  the minimization of the steady functional 
\begin{equation}\label{eq: steady_functional}
J_s (u,y) := \dfrac{1}{2} \left( \|u\|^2 + \|C\, y-z\|^2\right), 
\end{equation}
over the set of controlled steady states
\begin{equation}\label{def_M}
    M_s\coloneqq \left\{\left(u,y\right)\in U\times \mathbb{R}^n \ | \ Ay+Bu = 0\right\}.
\end{equation}
We denote by $V_{s}$ the value of the optimal steady cost, defined as
\begin{equation}\label{V_s def}
    V_{s} := \min \left\{ J_{s}(u,y),\  \text{s.t.}\ (u,y)\in M_s\, \right\}.
\end{equation}


\subsection{A turnpike result for control-constrained LQR}

In \cite{porretta2013long}, it is proved that, in the case where
$U=\mathbb{R}^m$, exponential turnpike holds under \textit{controllability} of $(A,B)$ and \textit{observability} of $(A,C)$. 
As we shall prove in \Cref{th_TURNPIKE} below, for the constrained control case (i.e. $U\subseteq \mathbb{R}^m$), the validity of a weaker version of the turnpike property follows from the \textit{detectability} of $(A,C)$, the \textit{invertibility} of $A$ and the $U$-\textit{stabilizability} of $(A,B)$ (see Definition \ref{definition_U_stab}). In \cite{grune2021relation}, for the constrained case, the validity of an even weaker version of the turnpike property (measure turnpike) was investigated under strong dissipativity assumptions and for the case when the steady optimum is at the interior of the admissible set. 
We point out that we shall not make the latter assumption, and thus, our turnpike result also applies to the case when the steady optimal control is on the boundary of the set of admissible controls $U$.
In the sequel, we will sometimes refer to the steady optimal control $\overline{u}$ and its corresponding state $\overline{y}$ as the turnpike.
Note that the steady functional $J_s$, hence also the turnpike, are independent of the final cost $g$. 

Let us first of all make precise the notion of constrained stabilizability.

\begin{definition}\label{definition_U_stab}
Let $A\in\mathcal{M}_{n\times n}(\mathbb{R})$, $B\in\mathcal{M}_{n,m}(\mathbb{R})$ and $U\subseteq \mathbb{R}^m$ be closed and convex. We say that $(A,B)$ is $U$-\textit{stabilizable} to a trajectory $y_{1}$, with some control $u_1\in L^2_{loc}(0,+\infty;U)$, solution to 
    \begin{equation}\label{target_controlled_ODE}
        \frac{d}{ds}y_1(s)=Ay_1(s)+Bu_1\hspace{1 cm} s\in  (0,+\infty),
    \end{equation}
if there exists a control $u\in L^2_{loc}(0,+\infty;U)$ for which the corresponding trajectory $y$ solves \eqref{target_controlled_ODE} with an initial datum $x\in \mathbb{R}^n$, and such that 
    \begin{enumerate}
        \item $u-u_1\in L^1(0,+\infty;\mathbb{R}^m)\cap L^2(0,+\infty;\mathbb{R}^m)$;
        \item $y-y_1\in L^1(0,+\infty;\mathbb{R}^n)\cap L^2(0,+\infty;\mathbb{R}^n)$;
        \item $\left\|u-u_1\right\|_{L^1\cap L^2}+\left\|y-y_1\right\|_{L^1\cap L^2}\leq K\left\|y_1\left(0\right)-x\right\|$,\\
        where $\left\|\cdot\right\|_{L^1\cap L^2}\coloneqq \left\|\cdot\right\|_{L^1\left(0,+\infty\right)}+\left\|\cdot\right\|_{L^2\left(0,+\infty\right)}$\; and $K=K(A,B,U)$.
    \end{enumerate}
We say that $(A,B)$ is $U$-\textit{stabilizable}, when it is $U$-\textit{stabilizable} to any trajectory $y_{1}$ solution to \eqref{target_controlled_ODE} with some control $u_1\in L^2_{loc}(0,+\infty;U)$ and for every initial datum $x\in \mathbb{R}^n$.
\end{definition}

We refer to Remark \ref{rmk_stab} for additional comments on the notion of $U$-stabilizability.

We may now give the statement of the turnpike result for the optimal control problem \eqref{eq: linear ODE}--\eqref{eq: functional}  with control constraints.

\begin{theorem}\label{th_TURNPIKE}
	Assume $(A,B)$ is $U$-stabilizable, $(A,C)$ is detectable and $A$ is invertible. Let $(\overline{y},\overline{u})$ be the unique pair in $M_s$ minimizing \eqref{eq: steady_functional}, and for any $T>0$, let $u_{_T}\in \mathcal{U}_T$ be an optimal control minimizing $J_{T,x}$ in \eqref{eq: functional}, and let $y_{_{T}}$ be its associated state trajectory, solution to \eqref{eq: linear ODE}. Then, for any $\varepsilon \in (0,1)$, there exists $\tau=\tau(A,B,C,U,x,z,g,\varepsilon)>0$ such that, if $T\geq 2\tau+1$, we have
\begin{equation}\label{epsturnpike}
    \left\|y_{_{T}}\left(t\right)-\overline{y}\right\|<\varepsilon
    \qquad \forall t\in \left[\tau,T-\tau\right].
\end{equation}
Furthermore, there exists $K=K(A,B,C,U,x,z,g)$, such that
\begin{equation}\label{lemma_unif_boun_2_intro}
    \int_0^T \left[\| u_{_{T}}(s)-\overline{u}\|^2 + \| y_{_{T}}(s)-\overline{y}\|^2\right] ds\leq K.
\end{equation}
\end{theorem}

We note that existence and uniqueness of a minimizer $(\overline{u},\overline{y})\in M_s$ for \eqref{eq: steady_functional}, and then, for the turnpike,
follows from the detectability of $(A,C)$ (see the Remark \ref{rmk: steady system - appendix} in Appendix \ref{sec:Proof of the turnpike property}).  
The uniqueness of the turnpike may fail,  for instance,  if the constraint set is nonconvex \cite{pighin2020nonuniqueness}, and also if the constraint set is convex but the pair $(A,C)$ is not detectable \cite{pighin2020turnpike2}.

\subsection{Main result}
Our main result establishes the connection between the turnpike property 
and the long-time behavior of the value function \eqref{eq: value function}. 
The latter is closely related to the value function associated to the corresponding infinite-horizon optimal control problem, that we define as
\begin{equation}\label{eq: W def}
        W(x) := \inf_{u\in \mathscr{A}_x} \ J_{\infty,x}(u) = \ \int_{0}^{\infty}\left[\dfrac{1}{2}\|u(s)\|^{2} + \dfrac{1}{2}\|C\, y(s) - z\|^{2} - V_s \right]\;ds,
    \end{equation}
    where, for each $u\in L^2_{loc}(0,+\infty;U)$, the function $y\in C([0,+\infty);\R^n)$ is the solution to \eqref{eq: linear ODE} in the time-interval $(0,+\infty)$, with initial condition $x$ and control $u$.
    Here, the set of admissible controls is 
    \begin{equation}
    \label{eq: admissible controls}
        \mathscr{A}_x\coloneqq \left\{u\in L^2_{\mbox{\tiny{loc}}}(0,+\infty;U) \ : \ \int_0^{\infty}\left|\dfrac{1}{2}\|u(s)\|^{2} + \dfrac{1}{2}\|C\, y(s) - z\|^{2} - V_s\right|ds < +\infty\right\}.
    \end{equation}

\begin{theorem}\label{th_turnpikeHJB}
    Assume $(A,B)$ is $U$-stabilizable, $(A,C)$ is detectable and $A$ is invertible. Let $z\in \R^n$ be given, and let $g:\R^n\to\R$ be a given locally Lipschitz function bounded from below. Let $V$ and $W$ be the value functions defined in \eqref{eq: value function} and \eqref{eq: W def} respectively.
    Then, the following statements hold true:
    \begin{enumerate}[label = (\roman*)]
        \item For any bounded set $\Omega\subset \mathbb{R}^n$, we have
        \begin{equation*}
        V(x,T) - V_{s}\, T \longrightarrow W(x) + \lambda, \quad \text{as} \ T\to \infty,\quad \mbox{uniformly in}\ x\in\Omega
        \end{equation*}
        where $V_{s}$ is the constant defined in \eqref{V_s def} 
        and the constant $\lambda$ is given by
        \begin{equation}\label{lambda in main thm}
            \lambda = \lim\limits_{T\to +\infty} V(\bar{y},T)-V_sT,
        \end{equation}
        where $\bar{y}$ is the state in the unique pair $(\bar{u},\bar{y})\in M_s$ minimizing \eqref{eq: steady_functional}.
        
        \item Moreover, $W(\cdot)$ is, up to an additive constant, the unique viscosity solution bounded from below to the stationary problem
        \begin{equation}\label{eq: ergodic HJB_thm}
            V_{s} + \max\limits_{u\in U} \left\{ -\nabla W(x)\cdot(Ax+Bu) - \frac{1}{2}\|u\|^{2} \right\} = \frac{1}{2}\|Cx-z\|^{2}
        \quad \quad x\in \mathbb{R}^{n}.
        \end{equation}
        In addition, the equation \eqref{eq: ergodic HJB_thm} with a different constant $c\neq V_s$ does not admit any viscosity solution bounded from below.
    \end{enumerate}
    
\end{theorem}

As it is well-known, the value function $V(x,T)$ defined in \eqref{eq: value function} is the unique viscosity solution to the following Cauchy problem\footnote{Note that, as defined in \eqref{eq: value function}, the value function $V$ depends on the initial condition $x$ and the time horizon $T$. We use the notation $\nabla V$ for the derivative of $V$ with respect to $x$, and $\partial_{T}V$ for its derivative with respect to $T$. These derivatives must be interpreted in the appropriate classical or viscosity sense depending on the situation, which will be specified in each situation.} 
\begin{equation}\label{eq: HJ_mainth}
\left\{
\begin{array}{l}
\partial_{T} V + \max\limits_{u\in U}\left\{ -\nabla V \cdot (Ax+Bu) - \frac{1}{2}\|u\|^{2} \right\} = \frac{1}{2}\|Cx-z\|^{2} \\
            \noalign{\vskip 2mm}
V(x,0) = g(x). 
\end{array}\right.
\end{equation}
A proof can be found in \cite[\S 5, Thm.5 and Thm.6]{kouhkouh2018dynamic} (see also \cite[Thm. 7.2.4]{cannarsa2004semiconcave}). It relies on the methods
in \cite[Section III.3]{bardi2008optimal} (see also \cite[Section 10.3]{evans2010partial}) and is based on the Dynamic Programming Principle. 

Our main result in Theorem \ref{th_turnpikeHJB} then describes the long-time behavior of the solution to this problem. 
Note that in the case where the final cost $g$ is a nonconvex function, even if it is very smooth, the solution to \eqref{eq: HJ_mainth} eventually loses regularity for $T$ sufficiently large (see Example \ref{lemma_nouniq}), and the equation must be interpreted in the viscosity sense \cite{crandall1992user,crandall1983viscosity,lions1982generalized} (see also \cite{bardi2008optimal,evans2010partial}).

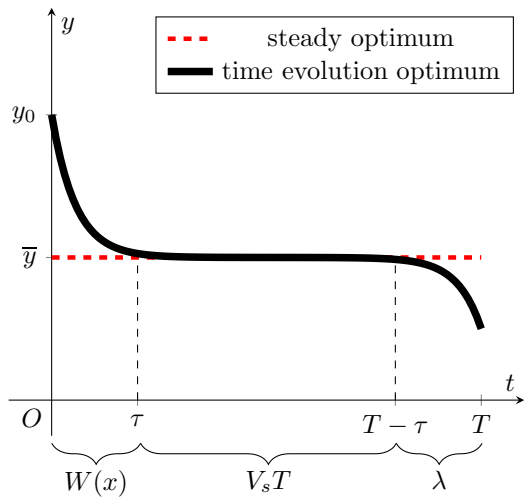
\begin{figure}
	\begin{center}
		\begin{tikzpicture}
		\begin{axis}[
		axis x line=middle, axis y line=middle,
		ymin=-1, ymax=11, ytick={4, 8}, yticklabels={$\overline{y}$, $y_0$}, ylabel=$y$,
		xmin=-1, xmax=11, xtick={2,8,10}, xticklabels={$\tau$,$T-\tau$,$T$}, xlabel=$t$,
		domain=0:10,samples=101, 
		]
		
		{\addplot [dashed, red,line width=0.06cm] {4};}
		\addlegendentry{steady optimum}
		\draw [dashed] (30,10) -- (30,50);
		\draw [dashed] (90,10) -- (90,50);
		{\addplot [black,line width=0.1cm] {4*exp(-1.8*\x)-2*exp(-1.8*(10-\x))+4};}
		\addlegendentry{time evolution optimum}
		\node [below left, black] at (10,10) {$O$};
		\end{axis}
		\draw [decorate,decoration={brace,amplitude=8pt,mirror,raise=0pt},yshift=0pt]
		(0.57,-0.1) -- (1.75,-0.1) node [black,below,midway,yshift=-0.24cm] {$W(x)$};
		\draw [decorate,decoration={brace,amplitude=8pt,mirror,raise=0pt},yshift=0pt]
		(1.75,-0.1) -- (5.15,-0.1) node [black,below,midway,yshift=-0.24cm] {$V_s T$};
		\draw [decorate,decoration={brace,amplitude=8pt,mirror,raise=0pt},yshift=0pt]
		(5.15,-0.1) -- (6.29,-0.1) node [black,below,midway,yshift=-0.24cm] {$\lambda$};
		\end{tikzpicture}
		\caption{Optimal state fulfilling the turnpike property and associated asymptotic decomposition of the value function.}
		\label{turnpike_asymptotic decomposition of the value function}
	\end{center}
\end{figure}

\begin{remark}\label{Rmk main thm}
In view of Theorem \ref{th_turnpikeHJB}, the value function $V(x,T)$ admits the asymptotic decomposition
$$
V(x,T) \sim  W(x) + V_s\, T + \lambda \qquad \text{as} \quad T\sim \infty,
$$
 where each term can be identified with one of the three stages in the turnpike optimal trajectory (see Figure \ref{turnpike_asymptotic decomposition of the value function}). 
 
\begin{enumerate}
    \item[1.] The term $W(x)$ represents the cost of stabilizing the trajectory from the initial state $x$ to the turnpike. Indeed, 
    in a large time interval, optimal strategies will spend most of the time close to the turnpike where, in view of  \eqref{eq: W def}, the running cost of the infinite-horizon problem equals zero.
    \item[2.] The term $V_s\, T$ corresponds to the running cost accumulated in the intermediate arc, where the time-evolution optima are close to the steady ones.
    \item[3.] The constant $\lambda$ represents the cost of leaving the turnpike
    in order to minimize the final cost $g$. 
    This final arc does not appear in the infinite  horizon problem, but
    it is always present in the finite horizon one, no matter how long the time-horizon is.
    Therefore, it has to be considered in the long-time decomposition of the value function.
    The way to single out this final arc from the rest of the trajectory is to consider the finite time horizon problem  taking $\overline{y}$ as initial state, so that the cost of reaching the turnpike is $0$ and then to subtract the cost during the transient arc $V_s\, T$ (see the definition of $\lambda$ in \eqref{lambda in main thm}).\\
    Let us also mention that such a constant $\lambda$ is usually refered to as \textit{the ergodic constant} (see \S \ref{sec: survey HJB}) which, roughly speaking, insures the existence of a set which attracts all controlled trajectories (see \cite{arisawa1997ergodic,arisawa1998ergodic}). And the \textit{turnpike} plays the role of such set where all the trajectories get as close as we wish regardless of their initial position. Moreover, the value of $\lambda$ represents the \textit{averaged cost} of being around the \textit{turnpike} (after omitting the first and second arcs). 
\end{enumerate}
\end{remark}

In case the control constraints are not imposed (i.e. $U=\mathbb{R}^m$), one can actually follow similar arguments as in \cite{porretta2013long} to prove exponential turnpike for the optimal control problem \eqref{eq: linear ODE}-\eqref{eq: functional}, under the assumptions of $(A,B)$ being stabilizable and $(A,C)$ detectable. 
We recall that exponential turnpike stands for the existence of constants $K,\mu>0$, independent of $T$, such that any optimal control-state trajectory $(y_T,u_T)$ satisfies
$$
    \|u^T(t)-\overline{u}\|+\|y^T(t)-\overline{y}\|\leq K\left[e^{-\mu t}+e^{-\mu \left(T-t\right)}\right] \quad \forall t\in [0,T].
$$ 
This allows us to deduce the same conclusions of Theorem \ref{th_turnpikeHJB} for the unconstrained case without assuming that $A$ is invertible.

 We note that in this unconstrained case $U=\mathbb{R}^m$, the value function $W$ for the infinite horizon problem, defined in \eqref{eq: W def}, is the viscosity solution to the Hamilton-Jacobi-Bellman equation
\begin{equation}\label{eq: ergodic HJB_thm unconstrained}
            V_s + \dfrac{1}{2}\|B^{*}\nabla W(x)\|^{2} - A\, x\cdot \nabla W(x) = \dfrac{1}{2}\|C\,x - z\|^{2}
        \quad \quad x\in \mathbb{R}^{n},
\end{equation}
and the solution can be given as a quadratic form using Riccati Theory.
We recall this result in the following proposition, and give a sketch of the proof in Appendix B for the sake of completeness.

\begin{proposition}\label{prop_rappr_riccati}
    Let $\left(\overline{u},\overline{y}\right)$ be the minimizer for $J_s$ defined in \eqref{eq: steady_functional}. Then,
    \begin{equation}\label{feedback}
        F(y)\coloneqq -B^*\widehat{E}\left(y-\overline{y}\right)+\overline{u}
    \end{equation}
    defines an optimal feedback law for $J_{\infty,x}$ defined in the right hand side of \eqref{eq: W def}, meaning that, for any $x\in\R^n$, the unique optimal control is given by
    \begin{equation}\label{opt_inf_1}
        u^{*}(s)= -B^*\widehat{E}\left(y^{*}(s)-\overline{y}\right)+ \overline{u},\qquad s\in (0,T)
    \end{equation}
    where $\widehat{E}$ is the unique symmetric positive semidefinite solution to the Algebraic Riccati Equation
    \begin{equation*}\label{ARE_3}
		-\widehat{E}A-A^*\widehat{E}+\widehat{E}BB^*\widehat{E}=C^*C \hspace{1 cm} \mbox{(ARE)}
		\end{equation*}
    and $y^{*}$ solves the closed loop equation
    \begin{equation*}\label{}
    \left\{ \begin{array}{ll}
    \frac{d}{ds}y^{*}(s) = \left(A\, + B\,F\right) y^{*}(s), &  s\in (0,\infty) \\
    y^{*}(0) = x.
    \end{array}\right.
    \end{equation*}
    Moreover, the value function $W$ defined in \eqref{eq: W def} is given by
    \begin{equation*}\label{infinity_value_function}
        W(x) =\dfrac{1}{2}\left(x-\overline{y}\right)^*\widehat{E}\left(x-\overline{y}\right)+(\overline{p},x-\overline{y})_{\mathbb{R}^n},
    \end{equation*}
    and is, up to an additive constant, the unique viscosity solution, bounded from below, to the equation \eqref{eq: ergodic HJB_thm}. 
\end{proposition}

\subsection{Related literature}\label{subsec:related results}
 In the definition of the value function $W$ in \eqref{eq: W def} associated to the infinite horizon problem, the normalisation of the running cost by subtracting the cost at the optimal steady state is indeed very often used in dissipativity-based approaches to turnpike properties, and in the analysis of receding-horizon optimal control. In this context, the normalised running cost is known as the \textit{supply rate}, see \cite[Thm.2]{angeli2011average}. 

In this direction, let us mention the seminal work of Willems \cite{willems1971least}, where existence results for infinite horizon optimal control problems are proved under frequency-domain and time-domain conditions. Moreover, a full characterisation of the solutions to the algebraic Riccati equation is also provided. We also refer to the books \cite{anderson2007optimal,kwakernaak1972linear} for the classical Riccati theory.

The relation between the value function and the turnpike property has been investigated by Gr\"une in \cite{grune2016approximation} for the discrete-time setting in the context of receding horizon control (see also \cite{grune2018turnpike,grune2016relation}). However. their motivations and also their conclusions are different to ours.  There, the author presents an iterative method based on the turnpike property, which is used to approximate the infinite horizon problem by a sequence of finite horizon ones. The Assumption 4.1 in \cite{grune2016approximation} represents, in fact, a discrete-time version of the turnpike property (see also Chapter 6 and 8 in \cite{grune2017nonlinear} for further details).
Other related results in connection with Model Predictive Control can be found in \cite{zanon2016tracking, zanon2018economic} and the references therein. 

Other recent works related to our study are \cite{faulwasser2020continuous} and \cite{breiten2020turnpike}. In \cite{faulwasser2020continuous}, the interplay between dissipativity and stability properties in continuous-time infinite horizon optimal control problems have been studied, and moreover, the question on the link between the latter problem and the associated HJB equation has also been raised. In \cite{breiten2020turnpike},  the authors analyse the receding horizon control problem with finite stages in infinite-dimensions  in order to study the corresponding infinite horizon problem. In particular, their Theorem 6.4 provides an error estimate between these two problems which is intimately related to the turnpike property. A recent survey of dissipativity methods in optimal control can be found in \cite{grune2021dissipativity}.

A similar result to Theorem \ref{th_turnpikeHJB}(i) also appears in \cite{kouhkouh2018dynamic}, where the terms $\lambda, V_{s}$ and $W(\cdot)$  are represented by a different approach: a Riccati operator in an augmented state space is introduced, taking into account the target $z$ as a state variable (with zero dynamics). A description of the asymptotic behavior of such Riccati operators is provided in \cite[Lemma 2]{kouhkouh2018dynamic}, which then is the key ingredient to determine the desired asymptotic behavior.

In the present manuscript, our analysis relies on the study of the value functions of both the finite-time and infinite-time horizon problems and their corresponding partial differential equations, together with the optimality conditions they satisfy. The latter tools (in particular the PDE characterization) are different from those usually encountered in Riccati theory and in dissipativity-based approaches, and help us to establish the link between the turnpike property as known in control theory with the asymptotic behavior of a certain class of PDEs, that is, a time-evolutive HJB equation and its corresponding ergodic version.

\subsection{Known results on long-time behavior for Hamilton-Jacobi equations}\label{sec: survey HJB}

Observe that the PDE in \eqref{eq: HJ_mainth} is a Hamilton-Jacobi equation of the form
\begin{equation}\label{eq: HJ general}
    \partial_{T}V + H(x,\nabla V) = \ell(x),\quad \text{in} \  \R^n\times (0,+\infty),
\end{equation}
where, in our case, the function $H: \R^n\times \R^n \to \R$, known as the \emph{Hamiltonian}, and the function $\ell :\R^n \to \R$ are given by
\begin{equation}
\label{intro - ham}
\begin{aligned}
    H(x,p)   = \; \max\limits_{u\in U}\left\{ -p\cdot (Ax+Bu) - \dfrac{1}{2}\|u\|^{2}\right\}\quad\text{and }\quad
    \ell(x)  =  \dfrac{1}{2}\|Cx-z\|^{2}.
\end{aligned}
\end{equation}

The long-time behavior for equations like \eqref{eq: HJ general}
has been widely studied in the literature,  especially in the flat torus, but also in more general settings,
e.g. \cite{barles2019large,barles2000large,fujita2006asymptotic,ishii2006asymptotic,ishii2008asymptotic,ishii2013short,roquejoffre2001convergence} and the references therein. 
Here, we deal with unbounded solutions in the whole space $\R^n$,
a scenario much less studied compared to the  case in the $n$-dimensional torus.
In the recent work \cite{barles2019large} it is proved, under suitable hypotheses on $H$, the existence of a constant $c\in\R$ such that
\begin{equation}\label{eq: long-time intro}
V(x,T) - c\, T \to \varphi (x), \qquad \text{as} \ T\to \infty,
\end{equation}
where $\varphi$ is called a \textit{corrector}, and is a viscosity solution to the \emph{stationary} Hamilton-Jacobi equation
\begin{equation}
    \label{eq: stationary equation}
    c + H(x, \nabla \varphi ) = \ell(x) , \qquad \text{in} \ \R^n.
\end{equation}
This is also called the \emph{ergodic problem} \cite{arisawa1997ergodic,arisawa1998ergodic,barles2006ergodic} and $c$ is known as the \textit{ergodic constant}.
For equations like \eqref{eq: stationary equation}, a solution is understood as a pair $(c,\varphi)$, where $c$ is a constant and $\varphi$ is a (continuous) viscosity solution to \eqref{eq: stationary equation}.

In Theorem \ref{th_turnpikeHJB}, we have obtained long-time asymptotics of the form \eqref{eq: long-time intro} for the solution to the HJB equation \eqref{eq: HJ_mainth}.
Although in \cite{barles2019large}, the unbounded case (space domain $\Omega =\R^n$) is treated, we point out that in our setting, the Hamiltonian does not satisfy all the assumptions required in \cite{barles2019large}. In particular, our function $H(x,p)$ defined in \eqref{intro - ham} is neither strictly convex in the $p$ variable nor coercive (in view of \eqref{intro - ham}, the Hamiltonian is coercive if and only if $B^*$ has a trivial kernel).

The paper is structured as follows. 
In \textbf{subsection \ref{sec:CTLT Hamilton-Jacobi}}, we prove a first result which is a direct consequence of the turnpike property, namely, the time-averages of the value function converge to the ergodic constant as the time horizon tends to infinity. In  \textbf{subsection \ref{subsection infinte time horizon pbm}}, we study the auxiliary infinite horizon optimal control problem introduced in \eqref{eq: W def}. Finally, in \textbf{subsection \ref{subsection: proof Thm 1.1 (1)}}, we give the proof of Theorem \ref{th_turnpikeHJB}.
In \textbf{Section \ref{sec:Conclusions and open problems}}, we sum up the conclusions of the paper and give a list of possible research lines.
Finally, for the reader's convenience and self-consistency of the paper, we include \textbf{Appendix A}, with the proof of the turnpike property stated in Theorem \ref{th_TURNPIKE}, and \textbf{Appendix B}, with some elements taken from classical Riccati Theory, without proof, which are necessary to justify Proposition \ref{prop_rappr_riccati}. 

\section{Infinite horizon problem and proof of Theorem \ref{th_turnpikeHJB}}
\label{sec: Proof main thm}

The proof of Theorem \ref{th_turnpikeHJB} relies on the turnpike property, which ensures that the optimal control for the problem \eqref{eq: value function} and its corresponding state trajectory remain close to the steady optima for any $t$ far away from $0$ and $T$.

\subsection{A first consequence of the turnpike property} 
\label{sec:CTLT Hamilton-Jacobi}
We start with a result which is a direct consequence of the turnpike property in Theorem \ref{th_TURNPIKE}. It ensures that the time-averages of the cost-functional $J_{T,x}(\cdot)$, evaluated in the optimal control $u_{_{T}}$, converge to the value of the steady optimal control problem as $T\to +\infty$.

\begin{proposition}\label{lemma_convaver}
Under the assumptions of Theorem \ref{th_turnpikeHJB}, let $V(x,T)$ be the value function defined in \eqref{eq: value function} and $V_s$ defined as in \eqref{V_s def}. Then, for any $x\in \R^n$, we have
\begin{equation}
\label{eq: limit time average}
    	\frac{1}{T}V (x,T)\underset{T\to +\infty}{\longrightarrow}V_s.
    \end{equation}
\end{proposition}

In order to prove the above proposition we need to rewrite the functional $J_{T,x}$ defined in \eqref{eq: functional} in a different way. Roughly speaking, we need the running cost to be centered at the turnpike. This is the content of the following Lemma.

\begin{lemma}\label{lemma_rappr_lower_bound}
Under assumptions of Theorem \ref{th_turnpikeHJB}, let $(\overline{u},\overline{y})$ be the steady optimal control-state pair for the functional $J_s$ defined in \eqref{eq: steady_functional} and $V_s := J_s (\overline{u},\overline{y})$. 
Then, for any $T>0$, $x\in\mathbb{R}^n$ and $u\in \mathcal{U}_T$, we have
    \begin{eqnarray}\label{eq}
    J_{T,x}(u) &=&T\, V_s  +\dfrac{1}{2} \int_0^T \left[\| u(s)-\overline{u}\|^2 + \| C \, \left(y (s)-\overline{y}\right)\|^2\right] ds\nonumber\\
    &\;& +\int_0^T \left[\left(\overline{u},u(s)-\overline{u}\right)_{\mathbb{R}^m}+\left(C\overline{y}-z,\, C \left( y (s)-\overline{y}\right)\right)_{\mathbb{R}^n}\right] ds+g(y(T))
    \end{eqnarray}
    and
    \begin{eqnarray}\label{ineq}
    J_{T,x}(u) &\geq&T\, V_s+\frac{1}{2} \int_0^T \left[\| u(s)-\overline{u}\|^2 + \| C \, \left(y(s)-\overline{y}\right)\|^2\right] ds\\
    &\;&+\left(\overline{p},x-y(T)\right)_{\mathbb{R}^n}+g(y(T)),\nonumber
    \end{eqnarray}
    where $\overline{p}\in \R^n$ is the optimal adjoint steady state (Lagrange multiplier)  and is independent of $T,x$ and $u$.
\end{lemma}
\begin{proof}[Proof of Lemma \ref{lemma_rappr_lower_bound}]
In view of the definition of $J_{T,x}$ in \eqref{eq: functional}, we can compute 
    \begin{eqnarray}\label{eq: value T and J_2}
    J_{T,x}(u)&=& \dfrac{1}{2} \int_0^T \left[\| u(s)-\overline{u}+\overline{u}\|^2 + \| C \, y (s)-C\overline{y}+C\overline{y}-z\|^2\right] ds + g(y(T))\nonumber\\
    &=&\dfrac{T}{2} \left[\| \overline{u}\|^2 + \|  C \, \overline{y}-z\|^2\right] +\dfrac{1}{2} \int_0^T \left[\| u(s)-\overline{u}\|^2 + \| C \, \left(y (s)-\overline{y}\right)\|^2\right] ds\nonumber\\
    &\;&\quad \quad +\int_0^T \left[\left(\overline{u},u(s)-\overline{u}\right)_{\mathbb{R}^m}+\left(C\overline{y}-z,\, C \left( y (s)-\overline{y}\right)\right)_{\mathbb{R}^n}\right] ds+g(y(T))\nonumber\\
    &=& T\, V_s  +\dfrac{1}{2} \int_0^T \left[\| u(s)-\overline{u}\|^2 + \| C \, \left(y (s)-\overline{y}\right)\|^2\right] ds\nonumber\\
    &\;& \quad \quad +\int_0^T \left[\left(\overline{u},u(s)-\overline{u}\right)_{\mathbb{R}^m}+\left(C\overline{y}-z,\, C \left( y (s)-\overline{y}\right)\right)_{\mathbb{R}^n}\right] ds+g(y(T)).
    \end{eqnarray}
We now focus on the term
    \begin{equation}\label{term to be computed}
        \int_0^T\left(C\, \overline{y}-z,\,  C \left( y(s)-\overline{y}\right)\right)_{\mathbb{R}^n}ds.
    \end{equation}
    We recall that the pair $(\overline{u},\overline{y})$ is optimal. Then, by using the convexity of $U$ and the invertibility of $A$, for any $u\in U$ we have the first order optimality condition\footnote{Indeed, consider the function $f:\left[0,1\right]\longrightarrow \mathbb{R}$, defined as $f\left(\delta\right)\coloneqq J_s\left(\left(\overline{u},\overline{y}\right)+\delta \left(u-\overline{u},y-\overline{y}\right)\right)$. Since $\left(\overline{u},\overline{y}\right)$ minimizes $J_s$, $f$ achieves its minimum at $\delta = 0$, whence $f^{\prime}\left(0\right)\geq 0$. Now, by \eqref{eq: steady_functional}, $f^{\prime}\left(0\right)=\left(\overline{u},u-\overline{u}\right)_{\mathbb{R}^m}+\left(C \, \overline{y}-z,C\, \left(y-\overline{y}\right)\right)_{\mathbb{R}^n}$. Now, the invertibility of $A$ guarantees the existence of an adjoint state $\overline{p}$ solving $0=A^*\overline{p}+C^*(C\, \overline{y}-z)$. Then, we can rewrite $f^{\prime}\left(0\right)=\left(\overline{u}+B^*\overline{p},u-\overline{u}\right)_{\mathbb{R}^m}$, whence (remembering that $f^{\prime}\left(0\right)\geq 0$) $\left(\overline{p},B\left(u-\overline{u}\right)\right)_{\mathbb{R}^n}\geq -\left(\overline{u},u-\overline{u}\right)_{\mathbb{R}^m}$.
    }
    \begin{equation}\label{steady_opt_cond}
    \left(\overline{p},B\left(u-\overline{u}\right)\right)_{\mathbb{R}^n}\geq -\left(\overline{u},u-\overline{u}\right)_{\mathbb{R}^m},
    \end{equation}
    where $0=A^*\overline{p}+C^*(C\, \overline{y}-z)$,
    which means that $\overline{u}$ is the projection of $-B^*\overline{p}$ onto $U$.
    On the other hand, the pairs $(u(\cdot),y(\cdot))$ and $(\overline{u},\overline{y})$ satisfy the equation in \eqref{eq: linear ODE}. Hence, we have
    \begin{equation}\label{difference_eq: linear ODE}
	\begin{cases}
	\frac{d}{ds}(y-\overline{y})=A(y-\overline{y})+B(u-\overline{u})\hspace{1 cm}& s\in  (0,T)\\
	y(0)-\overline{y}=x-\overline{y}.
	\end{cases}
	\end{equation}
Now, since $A$ is invertible the set $M_s\coloneqq \left\{\left(u,-A^{-1}Bu\right) \ | \ u\in U\right\}$ is well defined. Therefore, for a.e. $s\in (0,T)$, $u(s)$ (i.e. the time-evolution control evaluated at time $s$) is an admissible steady control in $M_s$. Hence, we are allowed to use \eqref{steady_opt_cond}, getting
\begin{equation*}
    \int_0^T  \left(\overline{p},B(u(s)-\overline{u})\right)_{\mathbb{R}^m} ds\geq -\int_0^T \left(\overline{u},u(s)-\overline{u}\right)_{\mathbb{R}^m} ds.
\end{equation*}
Employing the above inequality and \eqref{difference_eq: linear ODE} and taking into account $y(0)=x$,  the term \eqref{term to be computed} writes as:
    \begin{eqnarray}\label{lemma_rappr_lower_bound_eq6}
    \int_0^T\left(C\overline{y}-z,\, C \left( y(s)-\overline{y}\right)\right)_{\mathbb{R}^n}ds &=&\int_0^T\left(C^*\left(C\overline{y}-z\right),\, y(s)-\overline{y}\right)_{\mathbb{R}^n}ds\nonumber\\
    &=&-\int_0^T\left(\overline{p},\, A\left(y(s)-\overline{y}\right)\right)_{\mathbb{R}^n}ds\nonumber\\
    &=&-\int_0^T\left(\overline{p}, \frac{d}{ds}(y-\overline{y})-B(u-\overline{u})\right)_{\mathbb{R}^n}ds\nonumber\\
    &=&\left(\overline{p},y(0)-\overline{y}\right)_{\mathbb{R}^n}-\left(\overline{p},y(T)-\overline{y}\right)_{\mathbb{R}^n}\nonumber\\
    &\;&\quad \quad  + \int_0^T  \left(\overline{p},B(u-\overline{u})\right)_{\mathbb{R}^m} ds\nonumber\\
    &\geq&\left(\overline{p},x-y(T)\right)_{\mathbb{R}^n} -\int_0^T \left(\overline{u},u(s)-\overline{u}\right)_{\mathbb{R}^m} ds.
    \end{eqnarray}
    Finally, the conclusion follows by combining \eqref{eq: value T and J_2} and \eqref{lemma_rappr_lower_bound_eq6}.
\end{proof}

We can now give the proof of Proposition \ref{lemma_convaver}, which follows from Lemma \ref{lemma_rappr_lower_bound} and the $U$-stabilizability assumption in Definition \ref{definition_U_stab}.

\begin{proof}[Proof of Proposition \ref{lemma_convaver}]
   Let $x\in \R^n$ be fixed, and consider, for all $T>0$, the trajectory $y_1(t) = \overline{y}$ for all $t\in [0,T]$, which is associated to the constant control $u_1(t) = \overline{u}$ for all $t\in (0,T)$. By the $U$-stabilizability assumption (see Definition \ref{definition_U_stab}), there exist a control $u$ and its associated state trajectory $y$ such that
   $$
   \|u-\overline{u}\|_{L^1\cap L^2} + \|y-\overline{y}\|_{L^1\cap L^2} \leq K \| y_1(0) - x\|.
   $$
   Using the linearity of the dynamics, we can deduce from this that 
   $\|y(t) - \overline{y}\|\leq K\| y_1(0) - x\|$ for all $t\in [0,T]$.
   Then, using Lemma \ref{lemma_rappr_lower_bound} and the definition of the value function, we deduce that
   \begin{equation}\label{upper bound V}
   V(x,T) \leq J_{T,x} (u) \leq T\, V_s + C,
   \end{equation}
   for some constant independent of $T$.
   
   The lower bound follows from \eqref{ineq} applied to the optimal control $u_{_{T}}$, that is
   \begin{eqnarray}
    J_{T,x}(u_{_{T}}) &\geq&T\, V_s+\frac{1}{2} \int_0^T \left[\| u_{_{T}}(s)-\overline{u}\|^2 + \| C \, \left(y_{_{T}}(s)-\overline{y}\right)\|^2\right] ds\nonumber\\
    &\;&+\left(\overline{p},x-y_{_{T}}(T)\right)_{\mathbb{R}^n}+g(y_{_{T}}(T)).\nonumber
    \end{eqnarray}
    It then suffices to notice that the integral term is positive and, by the $T$-uniform bound of the optimal trajectory in \Cref{lemma_unif_bound}, the two last terms in the above inequality are bounded by a constant $K$ independent of $T$. Hence, one has
    \begin{equation}\label{lower bound V}
        V(x,T) = J_{T,x}(u_{_{T}}) \geq T V_{s} - K.
    \end{equation}
    The conclusion then follows after dividing the inequalities \eqref{upper bound V} and \eqref{lower bound V} by $T$ and taking the limit as $T\to +\infty$.
\end{proof}

We end this subsection with the following Lipschitz estimate, uniform in $T$, that is also consequence of the turnpike property and will be useful in the proof of Theorem \ref{th_turnpikeHJB}.

\begin{lemma}\label{lemma V lip estimate}
    Assume 
$(A,B)$ is $U$-stabilizable and $(A,C)$ is detectable. Let $V$ be the function defined in \eqref{eq: value function}. 
    Then, for any $M>0$, there exists a constant $K_M>0$ such that for all $T>0$ and all $x_1$ and $x_2$ in $\mathbb{R}^n$ satisfying $\|x_i\|\leq M$, we have
    \begin{equation*}
        \left|V(x_2,T)-V(x_1,T)\right|\leq K_M\left\|x_2-x_1\right\|.
    \end{equation*}
\end{lemma}
\begin{proof}
    We prove this Lemma by using the definition of $V(x,T)$ as minimal value of $J_{T,x}$. Let $u_{_{T},x_1}\in \mathcal{U}_T$ be an optimal control for $J_{T,x_1}$.
    
    \textit{Step 1} \ \textbf{Construction of stabilizing control}\\
    Since $(A,B)$ is $U$-stabilizable, there exists a control $\hat{u}\in L^2(0,T;U)$, such that
    \begin{equation}
        \left\|\hat{u}-u_1\right\|_{L^1\cap L^2(0,T)}+\left\|\hat{y}-y_1\right\|_{L^1\cap L^2(0,T)}\leq K\left(A,B,U\right)\left\|x_2-x_1\right\|,
    \end{equation}
    $\hat{y}$ being the solution to \eqref{eq: linear ODE}, with initial datum $x_2$ and control $\hat{u}$ and $\left\|\cdot\right\|_{L^1\cap L^2}\coloneqq \left\|\cdot\right\|_{L^1\left(0,T\right)}+\left\|\cdot\right\|_{L^2\left(0,T\right)}$.
    
    We have then
    \begin{equation}\label{diff_functionals}
        \left|J_{T,x_2}(\hat{u})-J_{T,x_1}(u_{_{T},x_1})\right|\leq K_M\left\|x_2-x_1\right\|,
    \end{equation}
    where $K_M$ is independent of $T>0$.
    
    \textit{Step 2} \ \textbf{Conclusion}\\
    For $i=1,2$, let $u_{_{T},x_i}$ be optimal controls for $J_{T,x_i}$ and let $\hat{u}$ defined as above for $u\coloneqq u_{_{T},x_1}$. Then, by definition of value function and \eqref{diff_functionals}
    \begin{eqnarray*}
    V(x_2,T)-V(x_1,T)&=&J_{T,x_2}(u_{_{T},x_2})-J_{T,x_1}(u_{_{T},x_1})\nonumber\\
    &\leq &J_{T,x_2}(\hat{u})-J_{T,x_1}(u_{_{T},x_1})\nonumber\\
    &\leq &K_M\left\|x_2-x_1\right\|.
    \end{eqnarray*}
    By the arbitrariness of $x_1$ and $x_2$, we obtain the desired Lipschitz property.
\end{proof}

\subsection{The infinite horizon linear-quadratic problem}
\label{subsection infinte time horizon pbm}

Here we introduce the auxiliary infinite time horizon optimal control problem announced in the introduction,
that allows us to compute the optimal cost of stabilizing the trajectory to the turnpike from the initial state.
For each $x\in\R^n$, the dynamics are determined by the same ODE in \eqref{eq: linear ODE}, in this case considering the time interval $(0,\infty)$:
\begin{equation}\label{eq: ODE infinity}
\begin{array}{ll}
\dot{y}(s) = A\, y(s) + B\, u(s), & s\in (0,\infty) \\
y (0) = x.
\end{array}
\end{equation} 
The set of admissible controls is $\mathscr{A}_x$  as defined in \eqref{eq: admissible controls} where $V_s = J_s(\overline{u},\overline{y})$ is the constant defined in \eqref{V_s def}. And the problem we shall consider is to minimize the cost functional
\begin{equation}\label{eq: functional_infinity}
    J_{\infty,x}(u):= \displaystyle\int_{0}^{\infty}\left[\dfrac{1}{2}\|u(s)\|^{2} + \dfrac{1}{2}\|C\, y(s) - z\|^{2} - V_s \right]\;ds,
\end{equation}
over the controls $u\in \mathscr{A}_x$. The value function for this problem is $W(x)$ as is defined in \eqref{eq: W def}. Note that the set of admissible controls is different for each $x$. In addition, since $(A,B)$ is $U$-stabilizable to $\overline{y}$, we deduce that it is nonempty for all $x$.

The following lemma follows directly from the definition of $\mathscr{A}_x$.
\begin{lemma}\label{lemma_rappr_infinity 1}
    Let $\left(\overline{u},\overline{y}\right)$ be the minimizer for $J_s$ defined in \eqref{eq: steady_functional}. For any $x\in\mathbb{R}^n$ and any control $u\in \mathscr{A}_x$, we denote by $y$ the solution to \eqref{eq: ODE infinity} with control $u$ and initial datum $x$. Then it holds
    $$
    u-\overline{u}\in L^2(0,+\infty;\mathbb{R}^m) \quad \text{and} \quad y-\overline{y}\in L^2(0,+\infty;\mathbb{R}^n).
    $$
    In addition, $\{y(t)\}_{t> 0}$ is bounded in $\R^n$ and satisfies
    $$
    y(t)\longrightarrow\overline{y} \quad \text{as} \quad t\to +\infty.
    $$
    The functional $J_{\infty,x}$ can be written as
    \begin{equation}\label{lemma_rappr_infinity_eq1}
    J_{\infty,x}(u) \geq \dfrac{1}{2} \int_0^{\infty} \left[\| u(s)-\overline{u}\|^2 + \| C \, \left(y(s)-\overline{y}\right)\|^2\right] ds +(\overline{p},x-\overline{y})_{\mathbb{R}^n}
    \end{equation}
    and it admits a minimizer $u^{*}$ in $\mathscr{A}_x$.
\end{lemma}

\begin{proof}
    \textit{Step 1} \  \textbf{Boundedness of $\left\{y(t)\right\}_{t>0}\subset \mathbb{R}^n$}\\
    Take any $u\in \mathscr{A}_x$ and let $y$ be the solution to \eqref{eq: ODE infinity}, with initial datum $x$ and control $u$. 
    By Lemma \ref{lemma_obs_conseq_dynamical} applied to $y-\overline{y}$, we have
    \begin{equation*}
        \|y(t)-\overline{y}\|^2\leq K\left[\|x-\overline{y}\|^2+\int_0^t \left[\| u(s)-\overline{u}\|^2 + \| C \, \left(y(s)-\overline{y}\right)\|^2\right] ds\right],
    \end{equation*}
    whence
    \begin{equation*}
        \dfrac{1}{2} \int_0^t \left[\| u(s)-\overline{u}\|^2 + \| C \, \left(y(s)-\overline{y}\right)\|^2\right] ds\geq \alpha \|y(s)-\overline{y}\|^2-K,
    \end{equation*}
    where $\alpha = \alpha(A,C)>0$ and $K=K(A,B,C,x,z)\geq 0$. Using the above inequality and adapting \eqref{ineq}, yields
    \begin{eqnarray*}
    J_{\infty,x}(u)&=& \lim_{t\to +\infty}\int_{0}^{t}\left[\dfrac{1}{2}\|u(s)\|^{2} + \dfrac{1}{2}\|C\, y(s) - z\|^{2} - V_s \right]\;ds\nonumber\\
    &\geq&\lim_{t\to +\infty}\left[\dfrac{1}{2} \int_0^t \left[\| u(s)-\overline{u}\|^2 + \| C \, \left(y(s)-\overline{y}\right)\|^2\right] ds\right.\nonumber\\
    &\;&\quad\quad\quad +(\overline{p},x-y(t))_{\mathbb{R}^n}\bigg]\nonumber\\
    &\geq&\limsup_{t\to +\infty}\left[\dfrac{1}{2} \int_0^t \left[\| u(s)-\overline{u}\|^2 + \| C \, \left(y(s)-\overline{y}\right)\|^2\right] ds\right.\nonumber\\
    &\;&\quad\quad\quad -K\left(1+\left\|y(t)-\overline{y}\right\|\right)\bigg]\nonumber\\
    &\geq&\limsup_{t\to +\infty}\left[\alpha \|y(t)-\overline{y}\|^2-K\left(\|y(t)-\overline{y}\|+2\right)\right]\nonumber\\
    &\geq&\dfrac{\alpha}{2}\limsup_{t\to +\infty} \|y(t)-\overline{y}\|^2-K.\nonumber\\
    \end{eqnarray*}
    Now, since $u\in \mathscr{A}_x$, the functional $J_{\infty,x}(u)<+\infty$. This, together with the above estimate, implies the boundedness of $\left\{y(t)\right\}_{t>0}\subset \mathbb{R}^n$.\\
    \textit{Step 2} \  \textbf{Proof of $u-\overline{u}\in L^2(0,+\infty;\mathbb{R}^m)$ and $y-\overline{y}\in L^2(0,+\infty;\mathbb{R}^n)$.}\\
    By Step 1, there exists a constant $K(u)\geq 0$, such that $\forall \, t>0$, $\left\|y(t)\right\|\leq K(u)$.
    By \eqref{ineq}, one gets
    \begin{equation}
    \label{lemma_rappr_infinity_eq3}
        \begin{aligned}
            &\int_{0}^{t}\left[\dfrac{1}{2}\|u(s)\|^{2} + \dfrac{1}{2}\|C\, y(s) - z\|^{2} - V_s \right]\;ds\\
            & \qquad \qquad \geq \dfrac{1}{2} \int_0^t \left[\| u(s)-\overline{u}\|^2 + \| C \, \left(y(s)-\overline{y}\right)\|^2\right] ds + (\overline{p},x-y(t))_{\mathbb{R}^n}
        \end{aligned}
    \end{equation}
    and using the above bound, for any $t>0$, we have
    \begin{equation*}
        \begin{aligned}
            &\int_{0}^{t}\left[\dfrac{1}{2}\|u(s)\|^{2} + \dfrac{1}{2}\|C\, y(s) - z\|^{2} - V_s \right]\;ds\\
            & \qquad \qquad \geq \dfrac{1}{2} \int_0^t \left[\| u(s)-\overline{u}\|^2 + \| C \, \left(y(s)-\overline{y}\right)\|^2\right] -K(u),
        \end{aligned}
    \end{equation*}
    whence, since $u\in \mathscr{A}_x$,
    \begin{eqnarray*}\label{}
    +\infty > J_{\infty,x}(u)&=&\lim_{t\to +\infty}\int_{0}^{t}\left[\dfrac{1}{2}\|u(s)\|^{2} + \dfrac{1}{2}\|C\, y(s) - z\|^{2} - V_s \right]\;ds\nonumber\\
    &\geq&\dfrac{1}{2} \int_0^{\infty} \left[\| u(s)-\overline{u}\|^2 + \| C \, \left(y(s)-\overline{y}\right)\|^2\right] ds-K(u),
    \end{eqnarray*}
    which in turn implies $u-\overline{u}\in L^2(0,+\infty;\mathbb{R}^m)$ and $C(y-\overline{y})\in L^2(0,+\infty;\mathbb{R}^n)$. Now, since the pair $(A,C)$ is detectable, adapting the techniques of the proof of Lemma \ref{lemma_obs_conseq_dynamical}, we have in fact $y-\overline{y}\in L^2(0,+\infty;\mathbb{R}^n)$.
    
    \textit{Step 3} \  \textbf{Proof of $y(t)\longrightarrow \overline{y}$ as $t\to +\infty$.}\\
    Now, since $y-\overline{y}\in L^2(0,+\infty;\mathbb{R}^n)$, there exists a sequence $t_m\to +\infty$, such that
    \begin{equation*}
        y(t_m)\underset{m\to +\infty}{\longrightarrow}\overline{y}.
    \end{equation*}
    By the above convergence and $u-\overline{u}\in L^2(0,+\infty;\mathbb{R}^m)$ and $C(y-\overline{y})\in L^2(0,+\infty;\mathbb{R}^n)$, for any $\varepsilon >0$, there exists $m_{\varepsilon}\in \mathbb{N}$ such that for every $m>m_{\varepsilon}$
    \begin{equation*}
        \left\|y(t_m)-\overline{y}\right\|< \varepsilon \hspace{0.3 cm}\mbox{and}\hspace{0.3 cm}\int_{t_m}^{+\infty} \left[\| u(s)-\overline{u}\|^2 + \| C \, \left(y(s)-\overline{y}\right)\|^2\right] ds<\varepsilon^2.
    \end{equation*}
    Then, by Lemma \ref{lemma_obs_conseq_dynamical}, for any $m>m_{\varepsilon}$ and for any $t>t_m$ we have
    \begin{equation*}
        \left\|y(t)-\overline{y}\right\|^2\leq K\left[\left\|y(t_m)-\overline{y}\right\|^2+\int_{t_m}^t \left[\| u(s)-\overline{u}\|^2 + \| C \, \left(y(s)-\overline{y}\right)\|^2\right] ds\right]<2K \varepsilon^2,
    \end{equation*}
    whence $y(t) \longrightarrow\overline{y}\,$ as $\,t\to +\infty$.
    
    \textit{Step 4} \  \textbf{Proof of \eqref{lemma_rappr_infinity_eq1}}\\
    The representation formula \eqref{lemma_rappr_infinity_eq1} is a consequence of \eqref{eq: functional_infinity}, \eqref{lemma_rappr_infinity_eq3}, $u-\overline{u}\in L^2(0,+\infty;\mathbb{R}^m)$, $\,y-\overline{y}\in L^2(0,+\infty;\mathbb{R}^n)$ and $y(t)\underset{t \to +\infty}{\longrightarrow}\overline{y}$. Existence of the minimizer follows from \eqref{lemma_rappr_infinity_eq1} and the Direct Method in the Calculus of Variations.
\end{proof}


Next we prove a local Lipschitz estimate for $W$ that will be used in the proof of Theorem \ref{th_turnpikeHJB}.

\begin{lemma}\label{lemma W lip estimate}
    Assume 
$(A,B)$ is $U$-stabilizable to $\overline{y}$ and $(A,C)$ is detectable and let $W$ be the function defined in \eqref{eq: W def}. 
    Then, for any $M>0$, there exists a constant $K_M>0$ such that
    \begin{equation*}
        \left|W(x_2)-W(x_1)\right|\leq K_M\left\|x_2-x_1\right\|,
    \end{equation*}
    for all $x_1$ and $x_2$ in $\R^n$ satisfying $\|x_i\|\leq M$.
\end{lemma}

\begin{proof}
The proof follows the techniques in the proof of Lemma \ref{lemma V lip estimate}.
\end{proof}

\subsection{Proof of Theorem \ref{th_turnpikeHJB}}
\label{subsection: proof Thm 1.1 (1)}

We are now in position to give the proof of Theorem \ref{th_turnpikeHJB}. We split the proof in three steps. In the first one, we prove the statement (i) of the Theorem about the convergence of the value function. In the step 2, we prove the uniqueness result for the solution of the Hamilton-Jacobi-Bellman equation associated to the infinite horizon problem. Finally, in step 3, we prove that, in the unconstrained case $U=\R^n,$ the value function for the infinite horizon problem is in $C^1(\R^n)$.

\begin{proof}[Proof of Theorem \ref{th_turnpikeHJB}]
\textit{Step 1:} \  \textbf{Convergence.}
Let $\Omega\subset \R^n$ be a bounded set. For any given $x\in\Omega$ and $T>0$, let $u_{_{T}}(\cdot)$ and $y_{_{T}}(\cdot)$ be an optimal control for problem \eqref{eq: linear ODE}--\eqref{eq: functional} and its corresponding state trajectory. Then, as a consequence of the DPP, for any $T>0$ we can write
\begin{equation}\label{proof Thm 1.1 DPP}
    V(x,T) = \dfrac{1}{2}\int_0^{\frac{T}{2}} \left[ \| u_{_{T}}(s)\|^2 + \|C\, y_{_{T}}(s)-z\|^2\right]ds +
    V\left( y_{_{T}}\left(\dfrac{T}{2}\right), \dfrac{T}{2} \right).
\end{equation}
Now, using Lemma \ref{lemma V lip estimate} and that, as a consequence of the turnpike property \eqref{epsturnpike}, $y_{_{T}}(T/2)\to \bar{y}$ as $T\to \infty$, we deduce that
\begin{equation*}\label{Proof Thm 1.1 second term}
\lim_{T\to \infty} \left| V\left( y_{_{T}}\left(\dfrac{T}{2}\right), \dfrac{T}{2} \right) - V\left( \bar{y}, \dfrac{T}{2} \right)\right|=0.
\end{equation*}
Hence, we have

\begin{eqnarray}
\lim_{T\to \infty}  V\left( y_{_{T}}\left(\dfrac{T}{2}\right), \dfrac{T}{2} \right) - \dfrac{T}{2} V_s &=& \lim_{T\to \infty}  \left[V\left( y_{_{T}}\left(\dfrac{T}{2}\right), \dfrac{T}{2} \right) -  V\left( \bar{y}, \dfrac{T}{2} \right)\right. \nonumber \\
& & \left.\quad \quad + V\left( \bar{y}, \dfrac{T}{2} \right) - \dfrac{T}{2}V_s \right]\nonumber  \\
&=& \lim_{T\to \infty} V\left( \bar{y}, \dfrac{T}{2} \right) - \dfrac{T}{2}V_s\; =:\; \lambda.
\label{proof Thm1.1 lambda}
\end{eqnarray}
The existence of this limit can be justified by proving that the function
$$
T\longmapsto V(\bar{y},T) - T\, V_s
$$
is decreasing and bounded from below.
Indeed, observe that if $u_{_{T}}$ is an optimal control for $J_{x,T}$, then for any $T'>T$, we can use the control 
$$
\hat{u} (s) := \left\{
\begin{array}{cc}
    \bar{u} & s\in (0,T'-T) \\
     u_{_{T}} (s) & s\in [T'-T,T') 
\end{array}
\right.
$$
to prove the monotonicity. The boundedness from below can be obtained from the turnpike property.

Let us now prove that
\begin{equation}\label{proof thm1.1 first half convergence}
\lim_{T\to+\infty} \dfrac{1}{2}\int_{0}^{\frac{T}{2}}\left[\|u_{_{T}}(s)\|^{2} + \|C\, y_{_{T}}(s) - z\|^{2}  \right]\;ds-\dfrac{T}{2}V_s = W(x).
\end{equation}
Let $u^*\in \mathscr{A}_x$ be the optimal control for the functional $J_{\infty,x}$ defined in \eqref{eq: W def} and $y^*$ its corresponding state trajectory. For any $T>0$, as a consequence of the DPP for the infinite horizon problem, we have
\begin{eqnarray}
W(x) &=& \int_0^{\frac{T}{2}} \left[ \dfrac{1}{2} \| u^*(s)\|^2 + \dfrac{1}{2}\|C\, y^*(s)-z\|^2 - V_s\right] ds +W\left(y^*\left(\frac{T}{2}\right)\right) \nonumber \\
&\leq & \dfrac{1}{2}\int_0^{\frac{T}{2}} \left[  \| u_{_{T}}(s)\|^2 + \|C\, y_{_{T}}(s)-z\|^2 \right] ds - \dfrac{T}{2}V_s +W\left(y_{_{T}}\left(\frac{T}{2}\right)\right).\label{proof Thm 1.1 DPP W}
\end{eqnarray}
Now, observe that by plugging $\Bar{y}$ in formula \eqref{lemma_rappr_infinity_eq1} in Lemma \ref{lemma_rappr_infinity 1}, one can easily see that $W(\Bar{y})=0$.
Then, using Theorem \ref{th_TURNPIKE} and that, by Lemma \ref{lemma W lip estimate}, the function $W(\cdot)$ is continuous, we deduce that
\begin{equation}\label{proof thm1.1 liminf}
    \liminf_{T\to+\infty} \dfrac{1}{2}\int_{0}^{\frac{T}{2}}\left[\|u_{_{T}}(s)\|^{2} + \|C\, y_{_{T}}(s) - z\|^{2}  \right]\;ds-\dfrac{T}{2}V_s \geq W(x).
\end{equation}

Using again the DPP, this time for the value function $V$, we obtain for any $T>0$:
\begin{eqnarray}
    V(x,T) &=& \dfrac{1}{2}\int_0^{\frac{T}{2}} \left[\|u_{_{T}} (s)\|^2 + \|C\,y_{_{T}}(s)-z\|^2\right] ds + V\left(y_{_{T}}\left(\dfrac{T}{2}\right),\dfrac{T}{2}\right) \nonumber \\
    &\leq & \dfrac{1}{2}\int_0^{\frac{T}{2}} \left[\|u^*(s)\|^2 + \|C\,y^*(s)-z\|^2\right] ds + V\left(y^*\left(\dfrac{T}{2}\right),\dfrac{T}{2}\right). \label{proof thm 1.1 limsup}
\end{eqnarray}
Using this time the DPP for $W$ (the first equality in \eqref{proof Thm 1.1 DPP W}), we can compute
$$
\dfrac{1}{2}\int_0^{\frac{T}{2}} \left[ \| u^*(s)\|^2 + \|C\, y^*(s)-z\|^2\right] ds =  W(x) + \dfrac{T}{2}V_s - W\left(y^*\left(\frac{T}{2}\right)\right).
$$
And combining this identity with \eqref{proof thm 1.1 limsup}, we obtain
\begin{eqnarray*}
& &\dfrac{1}{2}\int_0^{\frac{T}{2}} \left[\|u_{_{T}} (s)\|^2 + \|C\,y_{_{T}}(s)-z\|^2\right] ds - \dfrac{T}{2}V_s \leq   W(x) - W\left(y^*\left(\dfrac{T}{2}\right)\right) \\
& &\qquad \qquad \qquad \qquad + V\left(y^*\left(\dfrac{T}{2}\right),\dfrac{T}{2}\right) - V\left(y_{_{T}}\left(\dfrac{T}{2}\right),\dfrac{T}{2}\right).
\end{eqnarray*}
This inequality, together with $W(\bar{y})=0$, the Lipschitz continuity of $V$ from Lemma \ref{lemma V lip estimate} and the fact that, by the turnpike property and Lemma \ref{lemma_rappr_infinity 1}, we have that $y_{_{T}}(T/2)$ and $y^*(T/2)$ converge to $\bar{y}$ as $T\to \infty$, gives
$$
\limsup_{T\to+\infty} \dfrac{1}{2}\int_{0}^{\frac{T}{2}}\left[\|u_{_{T}}(s)\|^{2} + \|C\, y_{_{T}}(s) - z\|^{2}  \right]\;ds-\dfrac{T}{2}V_s \leq W(x).
$$
From this inequality and \eqref{proof thm1.1 liminf}, it follows \eqref{proof thm1.1 first half convergence}.

Finally, combining \eqref{proof Thm 1.1 DPP}, \eqref{proof Thm1.1 lambda} and \eqref{proof thm1.1 first half convergence} we obtain
\begin{equation}\label{conv_value_function_identification}
    V(x,T)-TV_s\underset{T\to +\infty}{\longrightarrow}W(x)+\lambda.
\end{equation}

\textit{Step 2:} \textbf{Uniqueness for the ergodic equation.}

The proof that $(V_s,W(\cdot))$ satisfies the equation \eqref{eq: ergodic HJB_thm} can be carried out by standard methods in optimal control theory. 
In the case where the inclusion $U\subset \mathbb{R}^m$ is strict (i.e. when we have constraints on the control), the function $W$ is not expected to enjoy $C^{1}$ regularity, and one has to use the theory of viscosity solutions \cite{crandall1983viscosity,crandall1992user}. We omit the proof since it follows exactly the arguments in \cite[Thm. 5]{kouhkouh2018dynamic} (which is an adaptation of \cite[Thm. 7.2.4]{cannarsa2004semiconcave} to the LQ setting), dropping the dependency on time. 

In order to prove that $W(x)$ is the unique (up to an additive constant) viscosity solution to \eqref{eq: ergodic HJB_thm} bounded from below, we argue by contradiction.
Let $c\in\R$, and let $W_1\in C(\R^n)$ be a bounded from below continuous function satisfying the equation 
$$
c + H(x,\nabla W_1) = \ell (x)
$$
in the viscosity sense. Here, the Hamiltonian $H$ and the function $\ell$ are defined as in \eqref{intro - ham}.
Observe that the function given by
$$
V_1(x,T) = c \, T + W_1(x)
$$
is a viscosity solution to the problem \eqref{eq: HJ_mainth} with initial condition $g(x) = W_1(x)$, which is bounded from below.
We can then deduce that $V_1(x,T)$ is actually the value function associated to the optimal control problem \eqref{eq: linear ODE}--\eqref{eq: functional} with final cost $g(x)=W_1(x)$. 
And since $W_1(\cdot)$ is bounded from below, we can use the statement (i) in Theorem \ref{th_turnpikeHJB} to deduce that
$$
\lim_{T\to +\infty} V_1(x,T) - V_s\, T = W(x) + \lambda, \qquad \text{for all}\ x\in \R^n,
$$
for some $\lambda\in\R$ depending on the final cost $W_1(\cdot)$. 
Hence, using the definition of $V_1(x,T)$ we obtain
$$
\lim_{T\to +\infty} W_1(x) + (c-V_s)\, T = W(x) + \lambda,
\qquad \text{for all}\ x\in \R^n.
$$
This implies that $c=V_s$ and also that $W_1(x)-W(x) = \lambda$, for all $x\in \R^n$.
\end{proof}

Let us finish this section with an illustrative example that shows why the value function $V(x,T)$ is not in general differentiable.
As we will see, for a suitable nonconvex final cost $g$, the global minimizer for $J_{T,x}$ with $x=0$ and $T$ sufficiently large is not unique. This implies in particular that the subdifferential of $V(\cdot, T)$ contains more than one element and hence $V(\cdot,T)$ is not differentiable at $0$ for $T$ sufficiently large (see  \cite[Theorem 7.4.17]{cannarsa2004semiconcave}, and further examples can be found in \cite[page 200]{cannarsa2004semiconcave}). In general, It can be shown (see \cite[Theorem 1.5.3]{cannarsa2004semiconcave}) that there exists a finite time horizon during which the solution is smooth, but afterwards, it develops singularities (see \cite[Theorem 1.5.6]{cannarsa2004semiconcave}).

\begin{example}\label{lemma_nouniq}
    Let us consider the optimal control problem \eqref{eq: linear ODE}--\eqref{eq: functional} with the pair of matrices $(A,B)$ being controllable and $C$ being any matrix.
    As a final cost, we consider the function
    $$
    g_\varepsilon(x) = \dfrac{1}{\varepsilon} [\| x\|^4 - \|x\|^2],
    $$
    where $\varepsilon>0$ will be chosen later.
    
    Our goal is to show that
    if $\varepsilon>0$ sufficiently small, the functional
    \begin{equation}\label{eq: functional_nouniq}
    J_{T,0} (u) \coloneqq \dfrac{1}{2} \int_0^T \left[\| u(s)\|^2 + \| C \, y (s)\|^2\right] ds + g_\varepsilon (y(T)), 
    \end{equation}
    admits (at least) two distinguished global minimizers whenever $T>2$.
    
    Let us first prove that, if $\varepsilon>0$ is sufficiently small, then for any $T>1$, the control $u\equiv 0$ is not optimal.

    Fix $x_1$ a minimizer of the function $g:\mathbb{R}^n\longrightarrow \mathbb{R}$ defined as $g(x)\coloneqq \left\|x\right\|^4-\left\|x\right\|^2$ and set
    \begin{equation*}
	\tilde{u}(s)=\begin{cases}
	0 \quad &s \in \ (0,T-1)\\
	u_1(t-T+1) \quad &s\in \ (T-1,T),
	\end{cases}
	\end{equation*}
	where $u_1$ is any control solving the controllability problem
	\begin{equation*}
    \begin{array}{ll}
    \dot{y_1}(s) = A\, y_1(s) + B\, u_1(s), & s\in [0,1] \\
    y_1 (0) = 0, \ y_1 (1) = x_1.
    \end{array}
    \end{equation*}
    Let $\tilde{y}$ be the solution to \eqref{eq: linear ODE},
    with control $\tilde{u}$. Since $x=0$ with control $u=0$ is a stationary point of \eqref{eq: linear ODE}, by uniqueness of solution  
    we have
    \begin{equation*}
	\tilde{y}(s)=\begin{cases}
	0 \quad &s \in \ (0,T-1)\\
	y_1(t-T+1) \quad &s\in \ (T-1,T),
	\end{cases}
	\end{equation*}
    Let us now evaluate the functional $J_{T,0}$ defined in \eqref{eq: functional_nouniq} at $\tilde{u}$ and compare it with the control $u\equiv 0$. Since $\min_{\R^n} g(x)<0$, we have
    \begin{eqnarray*}
    J_{T,0} (\tilde{u})
    &=&\dfrac{1}{2} \int_0^1 \left[\| u_1(s)\|^2 + \| C \, y_1 (s)\|^2\right] ds + \frac{1}{\varepsilon}\left[\left\|x_1\right\|^4-\left\|x_1\right\|^2\right]\nonumber\\
    &=&\dfrac{1}{2} \int_0^1 \left[\| u_1(s)\|^2 + \| C \, y_1 (s)\|^2\right] ds + \frac{1}{\varepsilon}\min_{\mathbb{R}^n}g \nonumber \\
    &<&  0 = J_{T,0}(0) 
    \end{eqnarray*}
    for a sufficiently small $\varepsilon$. This means that $u\equiv 0$ is not a global minimizer of \eqref{eq: functional_nouniq}.
    
    Finally, since the final cost $g_\varepsilon$ and the running cost in \eqref{eq: functional_nouniq} are continuous and bounded from below, then by the Direct Method in the Calculus of Variations, there exists a minimizer $u_{_{T}}$ of \eqref{eq: functional_nouniq}. Moreover, we have that $u_{_{T}}\neq 0$ if $T>1$. And since the initial condition of the admissible trajectories is $0$, then if we denote by $y_{_{T}}$ the optimal trajectory corresponding to $u_{_{T}}$, we have $-y_{_{T}}$ is the trajectory corresponding to $-u_{_{T}}$. Now, by definition of \eqref{eq: functional_nouniq}, $J_{T,0} (-u_{_{T}})=J_{T,0} (u_{_{T}})=\min_{\mathcal{U}_T}J_{T,0}$, whence $u_{_{T}}$ and $-u_{_{T}}$ are two distinguished global minimizers of \eqref{eq: functional_nouniq}.
\end{example}

\section{Conclusions and open problems}
\label{sec:Conclusions and open problems}
	
In this manuscript, we have studied the long time behavior of the value function associated to a finite-dimensional linear-quadratic optimal control problem with any target $z$, a general terminal cost $g$ and constrained controls. To do so, we have introduced an infinite-time horizon optimal control problem and studied its value function $W(x)$. This allows us to provide an asymptotic decomposition of the value function $V(T,x)$ for the original control problem with finite time horizon and which is of the form $W(x) + V_s\, T + \lambda$, where each of the terms corresponds to the cost of the optimal trajectory during one of the three stages of the turnpike strategy. 

We now present some open problems.

\subsection{Control problems governed by nonlinear state equations}
\label{subsec:Conclusions and open problems_Control problems governed by nonlinear state equations}

We formulate this for a special control problem. Let $A$ be an $n\times n$ symmetric positive definite matrix and let $f:\mathbb{R}\longrightarrow \mathbb{R}$ be an increasing nonlinearity of class $C^1$ and with $f(0)=0$. For a given time horizon $T>0$, an initial state $x$ in $\R^n$ and a control $u\in \mathcal{U}_T\coloneqq L^2(0,T;U)$ the corresponding trajectory $y (\cdot)$ solves
\begin{equation*}\label{eq: semilinear ODE}
\begin{array}{ll}
y'(s) + A\, y(s) + f\left(y(s)\right)= B\, u(s), & \text{for} \ s\in [0,T] \\
y (0) = x,
\end{array}
\end{equation*}
where the control operator is given by the matrix $B\in \mathcal{M}_{n,m} (\R)$ and the nonlinear term $f\left(y(s\right))=\left(f\left(y_1(s)\right),\dots,f\left(y_n(s)\right)\right)$.

The optimal control problem is to minimize, over the admissible controls $u\in L^2(0,T;U)$, the cost functional $J_{T,x} (u) \coloneqq \dfrac{1}{2} \int_0^T \left[\| u(s)\|^2 + \| C \, y (s)-z\|^2\right] ds $
where $C\in \mathcal{M}_n(\R)$ is a given matrix and $z\in \R^n$ is the prescribed running target. The value function 
is defined as $V (x,T) \coloneqq \inf_{u\in \mathcal{U}_T}  J_{T,x}(u)$.
In the same line as for the LQ problem treated in this manuscript, one can also introduce  the steady functional $J_{s} \left(\overline{u},\overline{y}\right) \coloneqq \dfrac{1}{2} \left[\| \overline{u}\|^2 + \| C \, \overline{y}-z\|^2\right]$
to be minimized over the subset of controlled steady states $M_s\coloneqq \left\{\left(\overline{u},\overline{y}\right)\in U\times \mathbb{R}^n \ | \ A\overline{y}+f\left(\overline{y}\right)=B\overline{u}\right\}$ and define $V_s\coloneqq \min_{M_s}J_s$. These kind of problems have been treated both in a finite dimensional framework \cite{trelat2015turnpike} and in a PDE framework \cite{PZ2,trelat2018steady,trelat2018integral,pighin2020turnpike}. Available results in the literature typically require smallness conditions on the running target.

By using the techniques developed in the above references, it is possible to get bounds on the space derivatives of the value function, which allow to apply the Ascoli-Arzel\`a Theorem as in the proof of Theorem \ref{th_turnpikeHJB}. For small targets, by using the turnpike results of \cite{pighin2020turnpike} and adapting the techniques of the present manuscript, we can deduce large time asymptotics of the value function as in Theorem \ref{th_turnpikeHJB}. However, for large targets, to the best of our knowledge, the turnpike theory is not complete. In particular, we cannot identify the limit as we do in \eqref{conv_value_function_identification}, because we do not have a result like
\begin{equation*}
\label{eq: limit time average_semil}
    	\frac{1}{T}V (x,T)\underset{T\to +\infty}{\longrightarrow}V_s,
\end{equation*}
which identifies the limit of the time-average of the value function as the value function for the steady problem. Note that, by adapting the techniques of \cite[Lemma 2.1, page 12]{pighin2020turnpike}, it is possible to prove that $\limsup_{T\to +\infty}\frac{1}{T}V (x,T)\leq V_s.$
But we are not able to prove the converse inequality
\begin{equation}
\label{eq: liminfgeq time average}
    	\liminf_{T\to +\infty}\frac{1}{T}V (x,T)\geq V_s.
\end{equation}
From a control perspective, the above inequality means that in time large there is no time-evolving strategy significantly better than the steady ones. Actually, in case the time evolving functional is restricted to time independent controls, \eqref{eq: liminfgeq time average} has been proved in \cite[section 4]{PZ2}, by $\Gamma$-convergence. However,
to the best of our knowledge, the above inequality is
unknown if the time-evolution functional is minimized over time dependent controls and it is an interesting open problem.

\subsection{Complete turnpike theory in constrained control}

Throughout our manuscript, we assumed $A$ invertible. It would be nice to get a turnpike result under constraints, without this assumption.

Furthermore, the rate of convergence of the time-evolution optima towards the steady ones should be investigated. An exponential bound could be obtained by applying the arguments of \cite{esteve2020turnpike} to $\tilde{u}\coloneqq u-\overline{u}$ and $\tilde{y}\coloneqq y-\overline{y}$.

\subsection{Characterization of the ergodic constant in a more general case}
\label{subsec:Conclusions and open problems_Characterization of the ergodic constant in a more general case}

In our setting, the constant $c$ in \eqref{eq: stationary equation} corresponds to $V_s$, which is the minimal value of the steady problem. This has been obtained as a consequence of the validity of the turnpike property for our problem. It would be interesting to generalize this characterization for more general problems, by using Hamilton-Jacobi techniques instead of turnpike theory.


\subsection{Hamilton-Jacobi equations with non-coercive Hamiltonian}
\label{subsec:Conclusions and open problems_Revise Hamilton-Jacobi literature to include the case of lack of coercivity and lack of Lipschitz property}

As we have anticipated, the function $p \longmapsto H(x,p)$, as defined in \eqref{intro - ham}, is not coercive whenever $B^*$ has a nontrivial kernel. This prevented us from using available results in the Hamilton-Jacobi literature. We have then employed turnpike theory to obtain long time behavior results in our context. 

\subsection{The infinite dimensional case}
\label{subsec:Conclusions and open problems_Tunrpike in infinite dimension}
It is well known that the turnpike property holds as well in the infinite dimensional case (see \cite{porretta2013long,PZ2,trelat2018steady,trelat2018integral,pighin2020turnpike}). In this setting, one can still associate to the infinite dimensional optimal control problem an analogue of the HJB equation that captures the evolution of the value function. Indeed, this can be handled for instance by means of the so-called \textit{Master} equation whose characteristics are of HJB type. Such equation appears in the context of Mean Field Games and its long time behavior was studied for instance in \cite{cardaliaguet2019long}. We refer also to \cite{bensoussan2015master, bensoussan2017interpretation}.

\appendix
\section{Proof of the turnpike property}
\label{sec:Proof of the turnpike property}

This appendix is devoted to the proof of the turnpike property stated in Theorem \ref{th_TURNPIKE}.
The main difficulty resides in the fact that we are considering the constrained control case.
In addition, we do not make the assumption of the steady optimal control $\overline{u}$ being at the interior of the control set $U$, which would make the proof much easier, and moreover would allow us to prove turnpike with an exponential rate.

We start by proving the following crucial Lemma which is a direct consequence of \cite[Remark 2.1]{porretta2013long} and we include its proof for self-consistency.

\begin{lemma}\label{lemma_obs_conseq_dynamical}
Assume $(A,C)$ is detectable and take $f\in L^2(0,T;\mathbb{R}^n)$. Then, there exists a constant $K=K\left(A,C\right)\geq 0$, independent of $T$ and $f$, such that for any $T\geq 1$ and for any $y$ solution to
\begin{equation}\label{ODE_nohomo}
    \frac{d}{ds}y=Ay+f\hspace{2.8 cm}  \mbox{in} \hspace{0.10 cm}(0,T),
\end{equation}
we have
\begin{equation}\label{obs_cons_dyna}
    \|y(t)\|^2+\int_0^T\| y\|^2ds\leq K\left[\|y(0)\|^2+\int_0^T\|C\, y\|^2ds+\int_0^T\|f\|^2 ds\right],
\end{equation}
for any $t\in [0,T]$.
\end{lemma}
\begin{proof}
    In the present proof, $K$ will denote a (sufficiently large) constant depending only on $(A,C)$.\\
    \textit{Step 1} \ \textbf{Decomposition into stable and antistable part}\\
    Following the notation of \cite{callier1995convergence}, $\mathscr{L}^{-}(A)$ and $\mathscr{L}^{0+}(A)$ denote resp. the $A$-invariant subspaces of $\mathbb{R}^n$ spanned by the generalized eigenvectors of $A$ corresponding to eigenvalues $\lambda$ of $A$ such that $\mbox{Re}(\lambda)<0$ and $\mbox{Re}(\lambda)\geq 0$. By linear algebra,
	\begin{equation*}
	    \mathbb{R}^n=\mathscr{L}^{-}(A)\oplus \mathscr{L}^{0+}(A),
	\end{equation*}
	where $\oplus$ stands for the direct sum. Then, let $y$ be a solution to \eqref{ODE_nohomo}. Denote by $y_1$ and $y_2$ resp. the projections of $y$ onto $\mathscr{L}^{-}(A)$ and $\mathscr{L}^{0+}(A)$. Then, $y=y_1+y_2$ and, for $i=1,2$,
	\begin{equation*}
	    \frac{d}{ds}y_i=Ay_i+f_i\hspace{2.8 cm}  \mbox{in} \hspace{0.10 cm}(0,T),
	\end{equation*}
	where $f_1$ and $f_2$ stand for resp. the projection of $f$ onto $\mathscr{L}^{-}(A)$ and $\mathscr{L}^{0+}(A)$.\\
	\textit{Step 2} \ \textbf{Estimate for the asymptotically stable part}\\
	We have
	\begin{equation*}
	    \frac{d}{ds}y_1=Ay_1+f_1\hspace{2.8 cm}  \mbox{in} \hspace{0.10 cm}(0,T),
	\end{equation*}
	All the eigenvalues of $L_A\hspace{-0.1 cm}\restriction_{\mathscr{L}^{-}(A)}$ are strictly negative, where we have denoted by $L_A$ the linear operator associated to the matrix $A$. Then,
	we have, for any $s\in  [0,T]$
	\begin{equation}\label{estimate_stable}
	    \|y_1(s)\|+\int_0^T\| y_1\|^2ds\leq K\left[\|y_1(0)\|+\|f_1\|_{L^2(0,T;\mathbb{R}^n)}\right]\leq K\left[\|y(0)\|+\|f\|_{L^2(0,T;\mathbb{R}^n)}\right],
	\end{equation}
	the constant $K$ depending only on $A$.\\
	\textit{Step 3} \ \textbf{An observability inequality in the time interval $[0,1]$}\\
	To proceed with the antistable part, we shall first prove the existence of an observability constant $K=K(A,C)\geq 0$, such that for any $\tilde{y}\in H^1(0,T;\mathscr{L}^{0+}(A))$ solution to
	\begin{equation*}
	    \frac{d}{ds}\tilde{y}=A\tilde{y}\hspace{2.8 cm}  \mbox{in} \hspace{0.10 cm}(0,1),
	\end{equation*}
	we have
	\begin{equation}\label{obs_ineq_01}
	    \left\|\tilde{y}\right\|_{L^{\infty}(0,1)}+\left\|\tilde{y}\right\|_{L^{2}(0,1)}\leq K\left\|C\tilde{y}\right\|_{L^{2}(0,1)}.
	\end{equation}
	To that end, define
	\begin{equation*}
	    \left\|\cdot\right\|_{a}:\mathscr{L}^{0+}(A)\longrightarrow \mathbb{R}^+,\hspace{0.3 cm}\left\|x\right\|_{a}\coloneqq\left\|\tilde{y}_x\right\|_{L^{\infty}(0,1)}+\left\|\tilde{y}_x\right\|_{L^{2}(0,1)}
	\end{equation*}
	and
	\begin{equation*}
	    \left\|\cdot\right\|_{b}:\mathscr{L}^{0+}(A)\longrightarrow \mathbb{R}^+,\hspace{0.3 cm}\left\|x\right\|_{b}\coloneqq\left\|C\tilde{y}_x\right\|_{L^{2}(0,1)},
	\end{equation*}
	where $\tilde{y}_x$ solves
	\begin{equation*}
    \begin{cases}
        \frac{d}{ds}\tilde{y}_x(s) = A\, \tilde{y}_x(s), &  s\in (0,1) \\
        \tilde{y}_x (0) = x.
    \end{cases}
    \end{equation*}
    Now, using that, being $(A,C)$ is detectable, all the modes in $\mathscr{L}^{0+}(A)$ are observable (see definition of detectability in \cite[at the bottom of page 232]{callier1995convergence}), we deduce that both $\left\|\cdot\right\|_{a}$ and $\left\|\cdot\right\|_{b}$ are norms on the subspace $\mathscr{L}^{0+}(A)$. Since $\mathscr{L}^{0+}(A)$ is finite dimensional, they are equivelent, whence \eqref{obs_ineq_01} follows.\\
	\textit{Step 4} \ \textbf{Estimate for the antistable part}\\
	By definition
	\begin{equation*}\label{antistable_part}
	    \frac{d}{ds}y_2=Ay_2+f_2\hspace{2.8 cm}  \mbox{in} \hspace{0.10 cm}(0,T),
	\end{equation*}
	Consider an arbitrary interval $[a,b]\subset [0,T]$, with length $\left|b-a\right|=1$. By step 3, we have
	\begin{equation*}
        \left\|y_2\right\|_{L^{\infty}(a,b)}^2+\left\|y_2\right\|_{L^{2}(a,b)}^2\leq K\left[\int_{a}^b\|C\, y_2\|^2 ds+\int_{a}^{b}\|f_2\|^2 ds\right].
    \end{equation*}
    
    On the one hand, due to the arbitrariness of $[a,b]$, this yields
    \begin{equation*}
	    \left\|y_2\right\|_{L^{\infty}(0,T)}^2\leq K\left[\int_{0}^T\|C\, y_2\|^2 ds+\int_{0}^T\|f_2\|^2 ds\right].
	\end{equation*}
	
	On the other hand,
	\begin{eqnarray*}
        \int_0^T\left\|y_2\right\|^2ds&\leq&\sum_{i=0}^{\lfloor{T}\rfloor}\int_i^{i+1}\left\|y_2\right\|^2ds+\int_{T-1}^T\left\|y_2\right\|^2ds\nonumber\\
        &\leq&K\sum_{i=0}^{\lfloor{T}\rfloor}\left[\int_{i}^{i+1}\|C\, y_2\|^2 ds+\int_{i}^{i+1}\|f_2\|^2 ds\right]\nonumber\\
        &\;&+K\left[\int_{T-1}^{T}\|C\, y_2\|^2 ds+\int_{T-1}^{T}\|f_2\|^2 ds\right]\nonumber\\
        &\leq& K\left[\int_{0}^T\|C\, y_2\|^2 ds+\int_{0}^T\|f_2\|^2 ds\right].
    \end{eqnarray*}
    
	Then, for any $t\in [0,T]$, we have
	\begin{eqnarray}\label{estimate_antistable}
        \|y_2(t)\|^2+\int_0^T\left\|y_2\right\|^2ds&\leq&K\left[\int_{0}^T\|C\, y_2\|^2 ds+\int_{0}^T\|f_2\|^2 ds\right]\nonumber\\
        &\leq&K\left[\int_{0}^T\|C\, y_2\|^2 ds+\int_{0}^T\|f\|^2 ds\right]\nonumber\\
        &\leq&K\left[\int_{0}^T\|C\, y\|^2 ds+\int_{0}^T\|C\, y_1\|^2 ds+\int_{0}^T\|f\|^2 ds\right]\nonumber\\
        &\leq&K\left[\|y(0)\|^2+\int_{0}^T\|C\, y\|^2 ds+\int_{0}^T\|f\|^2 ds\right],
    \end{eqnarray}
	where in the last inequality we have employed \eqref{estimate_stable}.\\
	\textit{Step 4} \ \textbf{Conclusion}\\
	Putting together \eqref{estimate_stable} and \eqref{estimate_antistable}, we conclude.
\end{proof}

\begin{remark}\label{rmk: steady system - appendix}
Observe that, assuming that$(A,C)$ is detectable, we have, for some $\beta = \beta(A,C)>0$, the inequality
\begin{equation}\label{obs_steady}
\|y_s\|^2\leq \beta\left[\|Ay_s\|^2+\|C\, y_s\|^2\right], \hspace{1 cm}\forall y_s\in \mathbb{R}^n,
\end{equation}
This is a consequence of inequality \eqref{obs_cons_dyna} applied to the trajectory $\tilde{y}(t)\coloneqq ty_s$, (see \cite{porretta2013long}).
The steady inequality \eqref{obs_steady} yields strict convexity of $J_s$ and hence uniqueness of the minimizer for the stationary optimal control problem \eqref{eq: steady_functional}.
\end{remark}

\begin{remark}\label{rmk_stab}
    In the Definition \ref{definition_U_stab}, by using $u-u_1\in L^1(0,+\infty;\mathbb{R}^m)$ and $y-y_1\in L^1(0,+\infty;\mathbb{R}^n)$ together with $\frac{d}{ds}\left(y\left(s\right)-y_1\left(s\right)\right)=A\left(y\left(s\right)-y_1\left(s\right)\right)+B\left(u\left(s\right)-u_1\left(s\right)\right)$, $s\in(0,+\infty)$, 
    the solution $y$ stabilizes towards $y_1$, i.e. $y\left(t\right)-y_1\left(t\right)\underset{t\to +\infty}{\longrightarrow}0$. 
    Note also that $U$-stabilizability follows from exact controllability under the control constraint $u\left(t\right)\in U$ (see e.g. \cite{brammer1972controllability,LMF,TAT}).
    In case $U=\mathbb{R}^m$, the $U$-stabilizability is equivalent to (unconstrained) exponential stabilizability of $(A,B)$ \cite[Remark 2.2 page 24]{BRC}.
\end{remark}

 We now prove the following result, which provides an upper bound $\left\|y_{_{T}}\right\|$ uniform in $T$, and also gives the inequality \eqref{lemma_unif_boun_2_intro} from Theorem \ref{th_TURNPIKE}.
\begin{lemma}\label{lemma_unif_bound}
    There exists $K=K(A,B,C,U,x,z,g)$ such that, for any $T\geq 1$ and for every $t\in [0,T]$, we have
    \begin{equation}\label{lemma_unif_boun_1}
        \left\|y_{_{T}}\left(t\right)\right\|\leq K
    \end{equation}
    and
    \begin{equation}\label{lemma_unif_boun_2}
        \int_0^T \left[\| u_{_{T}}(s)-\overline{u}\|^2 + \| y_{_{T}}(s)-\overline{y}\|^2\right] ds\leq K
    \end{equation}
\end{lemma}

\begin{proof}[Proof of Lemma \ref{lemma_unif_bound}]
    By Lemma \ref{lemma_obs_conseq_dynamical} applied to $y-\overline{y}$, we have
    \begin{equation*}
        \|y_{_{T}}(t)-\overline{y}\|^2\leq K\left[\|x-\overline{y}\|^2+\int_0^T \left[\| u_{_{T}}(s)-\overline{u}\|^2 + \| C \, \left(y_{_{T}}(s)-\overline{y}\right)\|^2\right] ds\right],
    \end{equation*}
    whence 
    \begin{equation*}
        \dfrac{1}{2} \int_0^T \left[\| u_{_{T}}(s)-\overline{u}\|^2 + \| C \, \left(y_{_{T}}(s)-\overline{y}\right)\|^2\right] ds\geq \alpha \|y_{_{T}}(s)-\overline{y}\|^2-K,
    \end{equation*}
    where $\alpha = \alpha(A,C)>0$ and $K=K(A,B,C,x,z)\geq 0$. Using the above inequality, together with Lemma \ref{lemma_rappr_lower_bound} (inequality \eqref{ineq}), yields
    
    \begin{eqnarray}\label{ineq_low}
    J_{T,x}(u_{_{T}})-TV_s&=& \int_{0}^{T}\left[\dfrac{1}{2}\|u(s)\|^{2} + \dfrac{1}{2}\|C\, y_{_{T}}(s) - z\|^{2} - V_s \right]\;ds\nonumber\\
    &\geq&\left[\dfrac{1}{2} \int_0^T \left[\| u_{_{T}}(s)-\overline{u}\|^2 + \| C \, \left(y_{_{T}}(s)-\overline{y}\right)\|^2\right] ds\right.\nonumber\\
    &\;&\quad\quad\quad +(\overline{p},x-y_{_{T}}(T))_{\mathbb{R}^n}+ g(y_{_{T}}(T))\bigg]\nonumber\\
    &\geq&\left[\dfrac{1}{2} \int_0^T \left[\| u_{_{T}}(s)-\overline{u}\|^2 + \| C \, \left(y_{_{T}}(s)-\overline{y}\right)\|^2\right] ds\right.\nonumber\\
    &\;&\quad\quad\quad -K\left(1+\left\|y_{_{T}}(T)-\overline{y}\right\|\right)\bigg]\nonumber\\
    &\geq&\left[\alpha \|y_{_{T}}(T)-\overline{y}\|^2-K\left(\|y_{_{T}}(T)-\overline{y}\|+2\right)\right]\nonumber\\
    &\geq&\dfrac{\alpha}{2} \|y_{_{T}}(T)-\overline{y}\|^2-K.
    \end{eqnarray}
    Now, by $U$-stabilizability,
    there exists a control $\hat{u}\in L^2(0,+\infty;U)$, such that
    \begin{equation*}
        \hat{u}-\overline{u}\in L^2(0,+\infty; \mathbb{R}^m)\cap L^1(0,+\infty, \mathbb{R}^m),\hspace{0.3 cm}\hat{y}-\overline{y}\in L^2(0,+\infty;\mathbb{R}^n)\cap L^1(0,+\infty; \mathbb{R}^n),
    \end{equation*}
    where $\hat{y}$ is the solution to \eqref{eq: linear ODE}, with initial datum $x$ and control $\hat{u}$. Therefore
    \begin{eqnarray}\label{ineq_up}
    J_{T,x}(u_{_{T}})-TV_s&\leq&J_{T,x}(\hat{u})-TV_s\nonumber\\
    &=&\dfrac{1}{2} \int_0^T \left[\| \hat{u}(s)-\overline{u}\|^2 + \| C \, \left(\hat{y} (s)-\overline{y}\right)\|^2\right] ds\nonumber\\
    &\;& +\int_0^T \left[\left(\overline{u},\hat{u}(s)-\overline{u}\right)_{\mathbb{R}^m}+\left(C\overline{y}-z,\, C \left( \hat{y} (s)-\overline{y}\right)\right)_{\mathbb{R}^n}\right] ds+g(\hat{y}(T))\nonumber\\
    &\leq &K,
    \end{eqnarray}
    where $=K(A,B,C,U,x,z,g)$.
    Hence, putting together \eqref{ineq_low} and \eqref{ineq_up}, we get
    \begin{eqnarray}
    \dfrac{\alpha}{2} \|y_{_{T}}(T)-\overline{y}\|^2-K&\leq&J_{T,x}(u_{_{T}})-TV_s\nonumber\\
    &\leq&J_{T,x}(\hat{u})-TV_s\leq K,\nonumber
    \end{eqnarray}
    i.e. the desired boundedness for $y_{_{T}}$.
    
    Moreover, by inequality \eqref{ineq}, the boundedness from below of $g$ together with \eqref{ineq_up}, we have
    \begin{align*}
    \int_0^T \left[\| u_{_{T}}(s)-\overline{u}\|^2 + \| C \, \left(y_{_{T}}(s)-\overline{y}\right)\|^2\right] ds\leq& \;J_{T,x}(u_{_{T}})-TV_s\nonumber\\
    & \quad +(\overline{p},-x+y_{_{T}}(T))_{\mathbb{R}^n}-g\left(y_{_{T}}\left(T\right)\right)\nonumber\\
    \leq& \;K.\\
    \end{align*}
    Then, inequality \eqref{lemma_unif_boun_2} follows from an application of Lemma \ref{lemma_obs_conseq_dynamical}. This finishes the proof.
\end{proof}
We now prove the validity of the turnpike property.

\begin{proof}[Proof of Theorem \ref{th_TURNPIKE}]
    Inequality \eqref{lemma_unif_boun_2_intro} has already been proved in Lemma \ref{lemma_unif_bound}. It remains to prove \eqref{epsturnpike}.
Throughout this proof, $K$ will always denote a (sufficiently large) constant depending only on $A$, $B$, $C$, $U$, $x$, $z$ and $g$.
    
    By \eqref{lemma_unif_boun_2_intro}, for any $\eta\in (0,1)$, there exists $\zeta=\zeta(A,B,C,U,x,z,g,\eta)>0$ such that, for all $T>\zeta$, we have
	\begin{eqnarray}
        \frac{1}{\zeta}\int_0^{\zeta} \left[\| u_{_{T}}(s)-\overline{u}\|^2 + \| y_{_{T}}(s)-\overline{y}\|^2\right] ds&\leq&\frac{1}{\zeta}\int_0^T \left[\| u_{_{T}}(s)-\overline{u}\|^2 + \| y_{_{T}}(s)-\overline{y}\|^2\right] ds\nonumber\\
        &\leq &\frac{K}{\zeta}<\eta^2\nonumber
    \end{eqnarray}
	and
	\begin{eqnarray}
        \frac{1}{\zeta}\int_{T-\zeta}^{T} \left[\| u_{_{T}}(s)-\overline{u}\|^2 + \| y_{_{T}}(s)-\overline{y}\|^2\right] ds&\leq&\frac{1}{\zeta}\int_{0}^T \left[\| u_{_{T}}(s)-\overline{u}\|^2 + \| y_{_{T}}(s)-\overline{y}\|^2\right] ds\nonumber\\
        &\leq &\frac{K}{\zeta}<\eta^2.\nonumber
    \end{eqnarray}
    By Integral Mean Value Theorem, for any $T\geq 1+2\zeta$, there exist $t_{T,1}\in [0,\zeta]$ and $t_{T,2}\in \left[T-\zeta,T\right]$, such that
    \begin{equation}\label{y_t-bary_small_1}
        \left\|u_{_{T}}\left(t_{T,1}\right)-\overline{u}\right\|^2 + \left\| y_{_{T}}\left(t_{T,1}\right)-\overline{y}\right\|^2=\frac{1}{\zeta}\int_0^{\zeta} \left[\| u_{_{T}}(s)-\overline{u}\|^2 + \| y_{_{T}}(s)-\overline{y}\|^2\right] ds<\eta^2
    \end{equation}
    and
    \begin{equation}\label{y_t-bary_small_2}
        \left\|u_{_{T}}\left(t_{T,2}\right)-\overline{u}\right\|^2 + \left\| y_{_{T}}\left(t_{T,2}\right)-\overline{y}\right\|^2=\frac{1}{\zeta}\int_{T-\zeta}^{T} \left[\| u_{_{T}}(s)-\overline{u}\|^2 + \| y_{_{T}}(s)-\overline{y}\|^2\right] ds<\eta^2.
    \end{equation}
    
    By the $U$-stabilizability, there exists a control $\tilde{u}\in L^2(t_{T,1},+\infty;U)$, such that
    $$\tilde{u}-\overline{u}\in L^1(t_{T,1},+\infty; \mathbb{R}^m)\cap L^2(t_{T,1},+\infty; \mathbb{R}^m)$$ 
    and its associated trajectory $\tilde{y}$, solution to \eqref{eq: linear ODE} with initial condition $y_{_T}(t_{T,1})$ satisfies
    $$\tilde{y}-\overline{y}\in L^1(t_{T,1},+\infty;\mathbb{R}^n)\cap L^2(t_{T,1},+\infty;\mathbb{R}^n),$$ 
    with estimates
    \begin{equation}\label{u-baruandy-bary}
    \begin{array}{l}
        \left\|\tilde{u}-\overline{u}\right\|_{L^1(t_{T,1},+\infty)\cap L^2(t_{T,1},+\infty)} \leq \gamma\left\|y_{_{T}}\left(t_{T,1}\right)-\overline{y}\right\|,\\
         \noalign{\vskip 2mm}
        \left\|\tilde{y}-\overline{y}\right\|_{L^1(t_{T,1},+\infty)\cap L^2(t_{T,1},+\infty)}\leq \gamma\left\|y_{_{T}}\left(t_{T,1}\right)-\overline{y}\right\|
    \end{array}
    \end{equation}
    the constant $\gamma$ depending only on $\left(A,B,U\right)$. Set
    \begin{equation}\label{eq: linear ODE_36}
        \begin{cases}
            \dot{\hat{y}}(s) = A\, \hat{y}(s) + B\, \hat{u}(s), &  s\in \left(0,T\right) \\
            \hat{y}\left(0\right) = x,
        \end{cases} \quad \quad
	    \hat{u}\left(s\right)\coloneqq
	    \begin{cases}
	        u_{_{T}}\left(s\right) \quad & s\in \left(0,t_{T,1}\right)\\
	        \tilde{u}\left(s\right) \quad & s \in \left(t_{T,1},t_{T,2}\right)\\
	        u_{_{T}}(s) \quad &  s\in \left(t_{T,2},T\right).
	    \end{cases}
	\end{equation}
    By using \eqref{u-baruandy-bary} and \eqref{eq: linear ODE_36}, we get
    \begin{equation}\label{u-baruandy-bary_inf}
        \left\|\hat{y}\left(t_{T,2}\right)-\overline{y}\right\|\leq \gamma\left\|y_{_{T}}\left(t_{T,1}\right)-\overline{y}\right\|,
    \end{equation}
    with $\gamma=\gamma\left(A,B,U\right)$.
    
    Now, let us define the functional
    \begin{equation}\label{eq: functional_Q}
    Q(u) \coloneqq \dfrac{1}{2} \int_{t_{T,1}}^{t_{T,2}} \left[\| u(s)\|^2 + \| C \, y (s)-z\|^2\right] ds, 
    \end{equation}
    defined for any $u\in L^2\left(t_{T,1},t_{T,2};U\right)$, where $y(\cdot)$ is the solution to \eqref{eq: linear ODE} with control $u$ and initial condition $y(t_{T,1}) = y_{_T} (t_{T,1})$.
    
    Let us estimate from above the following quantity
    \begin{equation*}
       \Lambda := \frac{1}{2}\int_{t_{T,1}}^{t_{T,2}} \left[\| u_{_{T}}(s)-\overline{u}\|^2 + \| C \, \left(y_{_{T}}(s)-\overline{y}\right)\|^2\right] ds.
    \end{equation*}
    
    By adapting the techniques of Lemma \ref{lemma_rappr_lower_bound} to the functional $Q$ in \eqref{eq: functional_Q}, we obtain the analogous version of \eqref{ineq}, which reads as
    \begin{align*}
       \Lambda \leq Q(u_{_{T}})-\left(t_{T,2}-t_{T,1}\right)V_s-\left(\overline{p},y_{_{T}}\left(t_{T,1}\right)-y_{_{T}}\left(t_{T,2}\right)\right)_{\mathbb{R}^n}.
        \end{align*}
Now, using the definition of $J_{T,x}$ and $Q$, along with the fact that $u_{_T}$ minimizes $J_{T,x}$, we deduce
    \begin{align*}
         \Lambda \  \leq & \  J_{T,x}(u_{_{T}})-\int_{[0,t_{T,1}]\cup [t_{T,2},T]} \left[\| u_{_{T}}(s)\|^2 + \| C \, y_{_{T}} (s)-z\|^2\right] ds-g\left(y_{_{T}}(T)\right)\\
         & \ -\left(t_{T,2}-t_{T,1}\right)V_s-\left(\overline{p},y_{_{T}}\left(t_{T,1}\right)-y_{_{T}}\left(t_{T,2}\right)\right)_{\mathbb{R}^n} \\
         \leq & \ J_{T,x}(\hat{u})-\int_{[0,t_{T,1}]\cup [t_{T,2},T]} \left[\| u_{_{T}}(s)\|^2 + \| C \, y_{_{T}} (s)-z\|^2\right] ds-g\left(y_{_{T}}(T)\right) \\
         & \ -\left(t_{T,2}-t_{T,1}\right)V_s-\left(\overline{p},y_{_{T}}\left(t_{T,1}\right)-y_{_{T}}\left(t_{T,2}\right)\right)_{\mathbb{R}^n}.
    \end{align*}
       Using the defintion of $Q$ in \eqref{eq: functional_Q} and the choice of the control $\hat{u}$ in \eqref{eq: linear ODE_36}, we get
        \begin{align}
       \Lambda \ \leq &  \  Q(\hat{u})+\int_{t_{T,2}}^T \left[\| C \, \hat{y} (s)-z\|^2 - \| C \, y_{_{T}} (s)-z\|^2\right] ds \nonumber \\
        &\ +g\left(\hat{y}(T)\right)-g\left(y_{_{T}}(T)\right) -\left(t_{T,2}-t_{T,1}\right)V_s-\left(\overline{p},y_{_{T}}\left(t_{T,1}\right)-y_{_{T}}\left(t_{T,2}\right)\right)_{\mathbb{R}^n}. \label{inequality Lambda 1}
        \end{align}

Now, noting that $u_{_T}$ and $\hat{u}$ coincide in the interval $(t_{T,2} , T)$, and that $T-t_{T,2} \leq \zeta$, we can use Gronwall's inequality to estimate
$$
\| \hat{y}(T) - y_{_T}(T)\| \leq C \| \hat{y} (t_{T,2}) - y_{_T} (t_{T,2}) \|,
$$
where $C$ is independent of $T$.
Moreover, using the local Lipschitz continuity of $g$ and \eqref{lemma_unif_boun_1}, we obtain
\begin{align}\label{estimate g}
|g\left(\hat{y}(T)\right)-g\left(y_{_{T}}(T)\right) | \leq K \| \hat{y} (t_{T,2}) - y_{_T} (t_{T,2}) \|.
\end{align}
        
Then, from the estimate \eqref{inequality Lambda 1}, an analogous version of the identity \eqref{eq} for the functional $Q$, combined with  \eqref{inequality Lambda 1},  \eqref{u-baruandy-bary} and \eqref{u-baruandy-bary_inf}, yields
        \begin{align*}
        \Lambda \  \leq & \  \int_{t_{T,1}}^{t_{T,2}} \left[\| \hat{u}(s)-\overline{u}\|^2 + \| C \, \left(\hat{y}(s)-\overline{y}\right)\|^2\right] ds \\
        & \  +\int_{t_{T,1}}^{t_{T,2}} \left[\left(\overline{u},\hat{u}(s)-\overline{u}\right)_{\mathbb{R}^m}+\left(C\overline{y}-z,\, C \left( \hat{y} (s)-\overline{y}\right)\right)_{\mathbb{R}^n}\right] ds\\
        &\  -\left(\overline{p},y_{_{T}}\left(t_{T,1}\right)-y_{_{T}}\left(t_{T,2}\right)\right)_{\mathbb{R}^n}+K\left[\left\|y_{_{T}}\left(t_{T,1}\right)-\overline{y}\right\|+\left\|y_{_{T}}\left(t_{T,2}\right)-\overline{y}\right\|\right]\\
         \leq & \  K\left[\left\|y_{_{T}}\left(t_{T,1}\right)-\overline{y}\right\|+\left\|y_{_{T}}\left(t_{T,2}\right)-\overline{y}\right\|\right],
    \end{align*}
  where $K$ is independent of $T$.

    Then, by Lemma \ref{lemma_obs_conseq_dynamical}, for any $s\in \left[t_{T,1},t_{T,2}\right]$,
    \begin{align}
        \|y_{_{T}}(s)-\overline{y}\|^2&\leq K\left\{\|y_{_{T}}\left(t_{T,1}\right)-\overline{y}\|^2+\int_{t_{T,1}}^{t_{T,2}} \left[\| u_{_{T}}(s)-\overline{u}\|^2 + \| C \,\left(y_{_{T}}(s)-\overline{y}\right)\|^2\right] ds\right\}\nonumber\\
        &\leq K\left[\left\|y_{_{T}}\left(t_{T,1}\right)-\overline{y}\right\|+\left\|y_{_{T}}\left(t_{T,2}\right)-\overline{y}\right\|\right]\leq K\eta,\nonumber
    \end{align}
    Finally, for any $\varepsilon >0$, setting $\eta =\frac{\varepsilon^2}{K}$ and $\tau(A,B,C,U,x,z,g,\varepsilon)=\zeta(A,B,C,U,x,z,g,\eta)$, we get the thesis.
\end{proof}

\section{Riccati theory and proof of Proposition \ref{prop_rappr_riccati}}
\label{sec:Riccati theory}



Although the proofs of our main results (Theorems \ref{th_TURNPIKE} and \ref{th_turnpikeHJB}) do not rely on the use of the classical Riccati theory, which  is not applicable to our case due to the constraints on the control, we note that in the unconstrained case $U=\mathbb{R}^m$, we may use the Riccati theory to obtain the value function $W(x)$ explicitly as a positively definite quadratic form. We recall that, following Theorem \ref{th_turnpikeHJB}, the value function $W(x)$ is the limiting profile of the asymptotic decomposition of the value function $V(T,x)$.

The proof of Proposition \ref{prop_rappr_riccati} is based on the following well-known Lemma, concerning the properties of the Algebraic Riccati Equation, and the corresponding Hamiltonian matrix
\begin{equation*}
\mbox{Ham}\coloneqq \begin{bmatrix}
A&-BB^*\\
-C^*C&-A^*
\end{bmatrix}.
\end{equation*}
One can realize that $\mbox{Ham}$ is the associated matrix to the optimality system \eqref{steady_OS_2_app}.

\begin{lemma}\label{lemma_iperbolicity}
	Assume $(A,B)$ is stabilizable and $(A,C)$ is detectable. Then,
	\begin{enumerate}
		\item there exists a unique symmetric positive semidefinite  solution to the Algebraic Riccati Equation
		\begin{equation}\label{ARE}
		-\widehat{E}A-A^*\widehat{E}+\widehat{E}BB^*\widehat{E}=C^*C \hspace{1 cm} \mbox{(ARE)}
		\end{equation}
		such that $A-BB^*\widehat{E}$ is stable, i.e. the real part of the spectrum $\mbox{Re}(\sigma(A-BB^*\widehat{E}))\subset (-\infty,0)$;
		\item set
		\begin{equation}\label{Lambda}
				\Lambda\coloneqq
				\begin{bmatrix}
				I_n&S\\
				\widehat{E}&\widehat{E}S+I_n
				\end{bmatrix},
		\end{equation}
		where $S$ is solution to the Lyapunov equation
		\begin{equation*}
		S(A-BB^*\widehat{E})^*+(A-BB^*\widehat{E})S=BB^*.
		\end{equation*}
		Then, $\Lambda$ is invertible and
		\begin{equation*}
		\Lambda^{-1}\mbox{Ham} \hspace{0.1 cm} \Lambda=\begin{bmatrix}
		A-BB^*\widehat{E}&0\\
		0&-(A-BB^*\widehat{E})^*.
		\end{bmatrix}
		\end{equation*}
		As a consequence, $\mbox{Ham}$ is invertible and its spectrum does not intersect the imaginary axis.
	\end{enumerate}
\end{lemma}

The first part of the above Lemma is Riccati theory (see, for instance, \cite[Fact 1-(a) and Fact 1-(f)]{callier1995convergence} or \cite{abou2012matrix}). The second part\footnote{
	\begin{equation*}
	\Lambda^{-1}=\begin{bmatrix}
	I_n+S\widehat{E}&-S\\
	-\widehat{E}&I_n.
	\end{bmatrix}
	\end{equation*}
} is taken from \cite[subsection III.B]{SMA}. We are now ready to prove Proposition \ref{prop_rappr_riccati}.

\begin{proof}[Proof of Proposition \ref{prop_rappr_riccati}]
    First of all, let us show that the minimization of $J_{\infty,x}$ is equivalent to the minimization of
    \begin{equation*}\label{}
    \begin{array}{cccl}
        \widehat{J}_{\infty,x}: & L^2_{\mbox{\tiny{loc}}}(0,+\infty;\R^m) & \longrightarrow & \mathbb{R}\cup \left\{+\infty\right\} \\
            \noalign{\vskip 2mm}
     & u & \longmapsto & \displaystyle\dfrac{1}{2} \int_0^{\infty} \left[\| u(s)-\overline{u}\|^2 + \| C \, \left(y(s)-\overline{y}\right)\|^2\right] ds,
     \end{array}
    \end{equation*}
    where $y$ is the solution to \eqref{eq: linear ODE}, with initial datum $x$ and control $u$.
    Let us show this by proceeding as in the proof of \eqref{ineq}, and concluding with Lemma \ref{lemma_rappr_infinity 1}.
    We first consider the finite horizon cost functional with final cost $g=0$, that is
    \begin{eqnarray}\label{eq: value T and J_2_app}
    J_{T,x}(u)&=& \dfrac{1}{2} \int_0^T \left[\| u(s)-\overline{u}+\overline{u}\|^2 + \| C \, y (s)-C\overline{y}+C\overline{y}-z\|^2\right] ds \nonumber\\
    &=&\dfrac{T}{2} \left[\| \overline{u}\|^2 + \|  C \, \overline{y}-z\|^2\right] +\dfrac{1}{2} \int_0^T \left[\| u(s)-\overline{u}\|^2 + \| C \, \left(y (s)-\overline{y}\right)\|^2\right] ds\nonumber\\
    &\;& +\int_0^T \left[\left(\overline{u},u(s)-\overline{u}\right)_{\mathbb{R}^m}+\left(C\overline{y}-z,\, C \left( y (s)-\overline{y}\right)\right)_{\mathbb{R}^n}\right] ds\nonumber\\
    &=& T\, V_s  +\dfrac{1}{2} \int_0^T \left[\| u(s)-\overline{u}\|^2 + \| C \, \left(y (s)-\overline{y}\right)\|^2\right] ds\nonumber\\
    &\;& +\int_0^T \left[\left(\overline{u},u(s)-\overline{u}\right)_{\mathbb{R}^m}+\left(C\overline{y}-z,\, C \left( y (s)-\overline{y}\right)\right)_{\mathbb{R}^n}\right] ds.
    \end{eqnarray}
Hence, one has
\begin{eqnarray}
    \frac{1}{2}\int_{0}^{T}\left[\|u(s)\|^{2} + \|C\,y(s) - z\|^{2} - V_{s}\right] \,ds
    &=& \dfrac{1}{2} \int_0^T \left[\| u(s)-\overline{u}\|^2 + \| C \, \left(y (s)-\overline{y}\right)\|^2\right] ds\nonumber\\
    &\;& +\int_0^T \left[\left(\overline{u},u(s)-\overline{u}\right)_{\mathbb{R}^m}+\left(C\overline{y}-z,\, C \left( y (s)-\overline{y}\right)\right)_{\mathbb{R}^n}\right] ds.\nonumber
    \end{eqnarray}
Then we focus on the term
    \begin{equation}\label{term to be computed_app}
        \int_0^T\left(C\, \overline{y}-z,\,  C \left( y(s)-\overline{y}\right)\right)_{\mathbb{R}^n}ds.
    \end{equation}
We recall that the pair $(\overline{u},\overline{y})$ satisfies the steady optimality system which reads as
    \begin{equation}\label{steady_OS_2_app}
    \begin{cases}
    0=A\overline{y}-BB^*\overline{p}\\
    0=A^*\overline{p}+C^*(C\, \overline{y}-z),
    \end{cases}
    \end{equation}
    with $\overline{u}=-B^* \overline{p}$. On the other hand, the pairs $(u(\cdot),y(\cdot))$ and $(\overline{u},\overline{y})$ satisfy the equation in \eqref{eq: linear ODE}. Hence, we have
    \begin{equation}\label{difference_eq: linear ODE_app}
	\begin{cases}
	\frac{d}{ds}(y-\overline{y})=A(y-\overline{y})+B(u-\overline{u})\hspace{1 cm}& s\in  (0,T)\\
	y(0)-\overline{y}=x-\overline{y}.
	\end{cases}
	\end{equation}
Then, using \eqref{steady_OS_2_app} and \eqref{difference_eq: linear ODE_app} and taking into account that $y(0)=x$ and $\overline{u}=-B^* \overline{p}$, we can compute the term \eqref{term to be computed_app} as follows:
    \begin{eqnarray}\label{lemmaraprV(x,T)evo_eq6}
    \int_0^T\left(C\overline{y}-z,\, C \left( y(s)-\overline{y}\right)\right)_{\mathbb{R}^n}ds &=&\int_0^T\left(C^*\left(C\overline{y}-z\right),\, y(s)-\overline{y}\right)_{\mathbb{R}^n}ds\nonumber\\
    &=&-\int_0^T\left(\overline{p},\, A\left(y(s)-\overline{y}\right)\right)_{\mathbb{R}^n}ds\nonumber\\
    &=&-\int_0^T\left(\overline{p}, \frac{d}{ds}(y-\overline{y})-B(u-\overline{u})\right)_{\mathbb{R}^n}ds\nonumber\\
    &=&\left(\overline{p},y(0)-\overline{y}\right)_{\mathbb{R}^n}-\left(\overline{p},y(T)-\overline{y}\right)_{\mathbb{R}^n}\nonumber\\
    &\;& + \int_0^T  \left(B^*\overline{p},u(s)-\overline{u}\right)_{\mathbb{R}^m} ds\nonumber\\
    &=&\left(\overline{p},x-y(T)\right)_{\mathbb{R}^n} -\int_0^T \left(\overline{u},u(s)-\overline{u}\right)_{\mathbb{R}^m} ds.
    \end{eqnarray}
    Finally, the conclusion follows by combining \eqref{eq: value T and J_2_app} and \eqref{lemmaraprV(x,T)evo_eq6} and then letting $T\to +\infty$ since from Lemma \ref{lemma_rappr_infinity 1} one has $u-\overline{u}\in L^{2}(0,+\infty;\R^{m})$, $y-\overline{y}\in L^{2}(0,+\infty;\R^{n})$ and $y(T) \to \overline{y}$.
    
    By \cite[Theorem 3.7 pages 237-238]{kwakernaak1972linear}, there exists a unique minimizer $u^{*}$ for $\widehat{J}_{\infty,x}$, given by \eqref{opt_inf_1} and
    \begin{equation*}
        \inf_{L^2_{\mbox{\tiny{loc}}}(0,+\infty;\R^m)} \widehat{J}_{\infty,x}(u)=\dfrac{1}{2}\left(x-\overline{y}\right)^*\widehat{E}\left(x-\overline{y}\right),
    \end{equation*}
    whence, by \eqref{eq: W def} and \eqref{lemma_rappr_infinity_eq1} which is now an equality,
    \begin{equation*}
        W(x) =\inf_{L^2_{\mbox{\tiny{loc}}}(0,+\infty;\R^m)} \widehat{J}_{\infty,x}(u)+\left(\overline{p},x-\overline{y}\right)_{\mathbb{R}^n}=\dfrac{1}{2}\left(x-\overline{y}\right)^*\widehat{E}\left(x-\overline{y}\right)+(\overline{p},x-\overline{y})_{\mathbb{R}^n},
    \end{equation*}
    as desired.
\end{proof}

\subsection*{Acknowledgement}
The authors are grateful to the referees for numerous remarks and suggestions which helped improve the first version of the manuscript.

\bibliography{turnpikeandHJB_references}
\bibliographystyle{siam}
\end{document}